\documentclass[12pt,a4paper]{article}
\usepackage{dsfont}
\usepackage{mathtext}
\usepackage[T2A]{fontenc} 
\usepackage[utf8]{inputenc}
\usepackage[russian,english]{babel} 
\usepackage{amsthm,amssymb,amsmath,amsfonts}
\usepackage{mathtools}
\usepackage{enumerate}
\usepackage[indentfirst]{titlesec} 
\usepackage[all]{xy}

\newtheorem{Th}{Theorem}[section]
\newtheorem{Prop}[Th]{Proposition}
\newtheorem{Lemma}[Th]{Lemma}
\newtheorem{Cor}[Th]{Corollary}
\newtheorem{Conj}[Th]{Conjecture}

\newtheorem*{Conj*}{Conjecture}
\newtheorem*{Quest*}{Question}
\newtheorem*{Th*}{Theorem}

\theoremstyle{definition} %%%%
\newtheorem{Def}[Th]{Definition}
\newtheorem{Remark}[Th]{Remark}
\newtheorem{Example}[Th]{Example}

\newcommand{\mult}{\operatorname{mult}}

\newcommand{\MC}{\mathds{C}}
\newcommand{\MP}{\mathds{P}}
\newcommand{\MQ}{\mathds{Q}}
\newcommand{\MZ}{\mathds{Z}}

\newcommand{\MA}{\mathds{A}}
\newcommand{\ON}{\operatorname}

\title{Birational geometry of del Pezzo fibrations with terminal quotient singularities}
\date{}
\author{Igor Krylov}
\begin{document}

\maketitle

\begin{abstract}
Del Pezzo fibrations appear as minimal models of rationally connected varieties. The rationality of smooth del Pezzo fibrations is a well studied question but smooth fibrations are not dense in moduli. Little is known about the rationality of the singular models. We prove birational rigidity, hence non-rationality, of del Pezzo fibrations with simple non-Gorenstein singularities satisfying the famous $K^2$-condition. We then apply this result to study embeddings of $\ON{PSL}_2(7)$ into the Cremona group.
\end{abstract}

\section{Introduction}
Rationality question is an old problem in algebraic geometry. The question was solved in dimension $2$ by Castelnuovo, who discovered an ``if and only if'' conditions for a surface to be rational. In dimension $3$ there is no such criterion.

Since rational varieties are rationally connected, we only need to study the rationality of rationally connected varieties. The minimal models of rationally connected varieties are Mori fiber spaces with rationally connected base. In dimension $2$ these are Hirzebruch surfaces and $\MP^2$, which are rational. In dimension $3$ these are Fano varieties, conic bundles over rational surfaces, and del Pezzo fibrations over the projective line. In this paper we study the latter class of varieties.

\begin{Th}
Let $\pi:X\to \MP^1$ be a del Pezzo fibration of degree $2$ with only $\frac{1}{2}(1,1,1)$-singularities (simplest terminal quotient singularities). Suppose $\ON{Pic}(X)=\MZ\oplus\MZ$ and suppose $X$ satisfies $K^2$-condition, that is $K_X^2$ is not in the interior of the Mori cone $\overline{NE}(X)$. Suppose also that fibers of $\pi$ containing singularities can be embedded into $\MP(1_x,1_y,1_z,2_w)$ as quartic cones $q_4(x,y,z)=0$. Then $X$ is birationally rigid, in particular not rational.
\end{Th}

\subsection{Birational geometry of del Pezzo fibrations}
We say that a del Pezzo fibration has degree $n$ if the general fiber is a del Pezzo surface of degree $n$. It is well known that all del Pezzo fibrations of degree $\geqslant 5$ are rational. The rationality of smooth del Pezzo fibrations of lower degree has been studied extensively, as a result a nearly complete solution to the problem has been obtained. See \cite{Alekseev} and \cite{Costya-dP4} for rationality and 
\cite{HT} for stable rationality of fibrations of degree $4$. Rationality for degrees $1$, $2$, and $3$ has been studied in \cite{Pukh123}, \cite{Grin1}, \cite{Grin2}, and \cite{Grin3}.

We say that a del Pezzo fibration is birationally rigid if it has only one Mori fiber space structure in its birational class (Definition 5.1). In particular it means that it is not birational to conic bundles and Fano varieties and therefore is not rational.

\begin{Th}[\cite{Pukh123}]
Let $\pi:X\to \MP^1$ be a smooth del Pezzo fibration of degree $1$ or $2$. Suppose $\ON{Pic}(X)=\MZ\oplus\MZ$ and suppose $X$ satisfies the $K^2$-condition, that is $K_X^2$ is not in the interior of Mori cone $\overline{NE}(X)$. Then the variety $X$ is birationally rigid, in particular $X$ is not rational.
\end{Th}

Unfortunately, Mori fiber spaces do not have to be smooth, in general they have terminal singularities. However in the case of conic bundles we only need to study smooth conic bundles. Given a conic bundle one can always find the standard model for it, which is smooth and minimal \cite{Sarkisov}. But it is not true for Fano varieties and del Pezzo fibrations. There are families of Fano threefolds whose minimal models have terminal quotient singularities and are unique, thus cannot be improved. The famous 95 families of index $1$ Fano varieties have this property (\cite{CPR}, \cite{ChP}), and more examples have been found recently (\cite{Okada14}, \cite{HO}, \cite{HZ}). 

Del Pezzo fibrations are in-between: there are good models, but they are not necessarily smooth or even Gorenstein \cite{Corti-models}. The good model of del Pezzo fibrations of degree $2$ are $2$-Gorenstein. Hence good models may not admit smoothing since the singularities of $2$-Gorenstein varieties are degenerations of the $\frac{1}{2}(1,1,1)$-singularity which does not admit smoothing. In particular it means that smooth del Pezzo fibrations are not dense in moduli. Varieties with only $\frac{1}{2}(1,1,1)$-singularities in a family of $2$-Gorenstein varieties are the equivalent of smooth varieties in a family of Gorenstein varieties.

It is expected that del Pezzo fibrations of degree $1$ also have models with good singularities, these models have to be $6$-Gorenstein. Therefore for del Pezzo fibrations of degree $1$ it is natural to consider $\frac{1}{2}(1,1,1)$ and $\frac{1}{3}(1,1,2)$-singularities \cite{Corti-models}. The following is a well-known and widely believed conjecture.

\begin{Conj}
Theorem $1.2$ be extended to del Pezzo fibrations of degree $2$ with only singularities of the type $\frac{1}{2}(1,1,1)$, and to del Pezzo fibrations of degree $1$ with only singularities of the type $\frac{1}{2}(1,1,1)$ and $\frac{1}{3}(1,1,2)$.
\end{Conj}

I prove a part of this conjecture in Theorem 1.1. Smooth del Pezzo surfaces of degree $2$ can all be embedded as quartics into $\MP(1_x,1_y,1_z,2_w)$, therefore the requirement of Theorem 1.1 that bad fibers can be embedded into $\MP(1,1,1,2)$ is quite natural. The equation of bad fiber is of the form $q_4(x,y,z)+wq_2(x,y,z)=0$, thus varieties we consider are special. The specialty condition is there only for technical reasons, there is no conceptual reason for this. The same technique should work to prove the conjecture.

\begin{Example}
Let $X$ be a hypersurface of bidegree $(4,l)$ in $\MP(1_x,1_y,1_z,2_w)\times\MP^1_{u,v}$. Then $X$ is a del Pezzo fibration of degree $2$. Since $-K_X$ is of bidegree $(1,2-l)$, $X$ satisfies the $K^2$-condition when $l\geqslant 2$. In particular it is not satisfied for $l=0$, that is when $X$ is a direct product. Let the equation of $X$ be $p_l(u,v)w^2=q_{4,l}(x,y,z;u,v)$, where $p$ and $q$ are generic polynomials of degree $l$ in $u,v$ and degree $4$ in $x,y,z$. Then the fibers containing singularities are quartic cones and $X$ satisfies the assumptions of Theorem 1.1.
\end{Example}

This example aligns quite well with \cite[Corollary~7.4]{Shokurov}. This corollary states that a fibrations with sufficiently many singular birationally nonsmoothable fibers is weakly rigid. Here the fiber over $t\in \MP^1$ is birationally nonsmoothable if there are no fiberwise maps to a fibration $\pi_Y: Y\to \MP^1$ such that a fiber of $\pi_Y$ over $t$ is smooth. It is unclear how to verify this condition, but we believe that $X$ satisfies it for sufficiently big $l$.

The corollary also states that a del Pezzo fibration is weakly rigid if the $K^2$-condition is preserved under fiberwise maps. It is not known, how to check this. However there are varieties which satisfy this condition trivially: they do not admit fiberwise maps which are not isomorphisms \cite[Theorem~1.1]{RCFT}.

We can apply Theorem 1.1 to study finite subgroups of Cremona group of rank three.

\subsection{Conjugacy classes of simple subgroups of Cremona group}
The \emph{Cremona group} $\ON{Cr}_n$ of rank $n$ is a group of birational transformations of $\MP^n$. It is clear that $\ON{Cr}_1=\ON{PGL}_2(\MC)$ and it is classically known that $\ON{Cr}_2=\langle \ON{PGL}_3(\MC), \sigma \rangle$, where 
\begin{align*}
\sigma:(x:y:z)\mapsto \Big(\frac{1}{x}:\frac{1}{y}:\frac{1}{z}\Big).
\end{align*}
Unfortunately, there is little hope for finding generators of the Cremona group over $\ON{PGL}_{n+1}(\MC)$ for $n\geqslant 3$. The Cremona group of rank $n$ is a lot more complicated and there are too many generators over $\ON{PGL}_{n+1}(\MC)$ to study them effectively for $n\geqslant 3$ \cite{Pan}.

Recently other ways to study the Cremona group have been considered. The topology of $\ON{Cr}_n$ have been studied in \cite{Blanc}. We could also study subgroups of the Cremona group, for example, it has been shown that $\ON{Cr}_n$ is not a simple group \cite{CL}.

In this paper we focus on finite subgroups. Finite subgroups of the $\ON{Cr}_2$ have been classified up to conjugation in \cite{Cr2}. To acquire the classification one can transfer this algebraic problem to the geometric language. 

\begin{Lemma}[{\cite[Proposition~1.2]{GFanoI}}]
Let $X$ be a rationally connected variety and let $G$ be a finite subgroup of $\ON{Bir} X$. Then there exists a variety $Y$ with a regular $G$-action and a $G$-equivariant birational map $\varphi:X\dasharrow Y$ such that
\begin{itemize}
	\item $Y$ is terminal and $G\MQ$-factorial, that is $G$-invariant divisors on $Y$ are $\MQ$-factorial,
	\item there is a $G$-equivariant map $\pi:Y\to Z$ to a variety of lower dimension such that a generic fiber of $Y$ is a Fano variety and a relative $G$-invariant Picard rank is $\rho^G(Y/Z)=1$.
\end{itemize}
\end{Lemma}

Variety $Y$ is called $G$-Mori fiber space. The classification of subgroups of $\ON{Cr}_n$ up to conjugation is equivalent to the classification of the rational $G$-Mori fiber space up to $G$-equivariant birational equivalence.

For a while very little was known about finite subgroups of Cremona group of higher rank. Serre even asked if every group can be embedded into the $\ON{Cr}_3$ \cite[Question~6.0]{Serre}. However several striking results have been found recently, among which is the following theorem of Prokhorov answering the question of Serre.

\begin{Th}[{\cite[Theorem~1.3]{Cr3}}]
Let $G$ be a finite simple non-abelian group. Then $\ON{Cr_3}(\MC)$ has a subgroup isomorphic to $G$ if and only if $G$ is one of the following groups: $\mathcal{A}_5$, $\ON{PSL}_2(7)$, $\mathcal{A}_6$, $\mathcal{A}_7$, $\ON{PSL}_2(8)$, or $\ON{PSU}_4(2)$.
\end{Th}

The theorem has been proven using the same idea: transfer the problem to geometry and work with good models of the action. This is of course a lot harder to do in dimension $3$ than in dimension $2$. Already the rationality question for $3$-fold Mori fiber spaces has several unresolved aspects.

There is little hope to classify finite subgroups of $\ON{Cr}_3$, but we should be able to classify finite simple subgroups up to conjugacy. Good models for actions of $\mathcal{A}_7$, $\ON{PSL}_2(8)$, or $\ON{PSU}_4(2)$ have already been classified, their models are $G$-Fano varieties since these groups cannot be embedded into $\ON{Cr}_2$. The Fano varieties with group actions is the most studied case so far: \cite{CC12}, \cite{CC13}, \cite{GFanoI}, \cite{GFanoII}, and \cite{GFano}. For $G$-conic bundles all we know is that the process of standardizing a conic bundle can be carried over equivariantly \cite{Avilov}. In this paper we focus on $G$-del Pezzo fibrations. The groups which could act on del Pezzo fibrations must be embedded into $\ON{Cr}_2$, these are: $\mathcal{A}_6$, $\ON{PSL}_2(7)$, and $\mathcal{A}_5$. The $\mathcal{A}_6$-del Pezzo fibrations have  already been classified.

\begin{Th}[{\cite[Appendix~B.]{CC13}}]
Let $X$ be an $\mathcal{A}_6$-del Pezzo fibration over $\MP^1$, then $X\cong\MP^2\times\MP^1$.
\end{Th}

For $\ON{PSL}_2(7)$ there is a series of families of del Pezzo fibrations and most of them have moduli, which makes the study and results more exciting compared to the case of  $\mathcal{A}_6$. %For $\mathcal{A}_5$ the problem is a lot harder, because $\mathcal{A}_5$ may act nontrivially on $\MP^1$.

\begin{Example}
Let $Y_n=\Big(\MC^6\setminus Z(I)\Big)/(\MC^*)^2$, where $I=\langle x,y,z,w \rangle\cap\langle u,v\rangle$ and the action of $(\MC^*)^2$ is given by the matrix:
\begin{align*}
\left(\begin{array}{ccccccc}
u&v&x&y&z&w\\
0&0&1&1&1&2\\
1&1&0&0&0&-n\end{array}\right)
\end{align*}
The variety $Y_n$ is a $4$-dimensional toric variety with $\ON{Cox}(Y_n)=\MC[x,y,z,w,u,v]$, the grading of this ring is defined by the matrix $A$. Clearly, $Y_0$ is a direct product $\MP^1\times\MP(1,1,1,2)$.
Let $X_n$ be the hypersurface of bidegree $(4,0)$ in $Y_n$ given by the equation
\begin{align*}
a_{2n}(u,v)w^2=x^3y+y^3z+z^3x,
\end{align*}
where $a_{2n}$ is a homogeneous polynomial of degree $2n$ without multiple roots. The projection
\begin{align*}
(u:v:x:y:z:w)\mapsto(u:v)
\end{align*} 
defines a structure of a del Pezzo fibration of degree $2$ on $X_n$. 

It is easy to see that $X_n$ admits the action of the group $\ON{PSL}_2(7)$ because right-hand side of the equation of $X_n$ is a polynomial defining the Klein quartic. Variety $X_n$ has $2n$ points of the type $\frac{1}{2}(1,1,1)$ and has no other singularities. Thus $X_n$ is terminal and $\MQ$-factorial.
\end{Example}

We develop a new technique to prove the following theorem and to reprove Theorem 1.7.

\begin{Th}[{\cite[Conjecture~3.3]{Hamid}}]
Let $X$ be a $\ON{PSL_2(7)}$-del Pezzo fibration over $\MP^1$. Let $S_2$ be the double cover of $\MP^2$ branched over the Klein quartic. Then
\begin{itemize}
	\item the generic fiber is $\MP^2$ and $X\cong\MP^2\times\MP^1$,
	\item or the generic fiber is $S_2$ and $X\cong X_n$.
\end{itemize}
\end{Th}

We also want to know which $X_n$ contribute to the Cremona group, that is which ones are rational. It is easy to see that $X_0$ and $X_1$ are rational.

\begin{Conj}[{\cite[Conjecture~3.5]{Hamid}}]
Varieties $X_n$ are not rational for $n\geqslant 2$.
\end{Conj}

%Note that $X_1$ and $X_0$ are unique, while $X_n$ have moduli for $n\geqslant 2$. Varieties $X_1$ and $\MP^2\times\MP^1$ are $\ON{PSL}_2(7)$-equivariantly birationally rigid, thus the corresponding embeddings into Cremona group are conjugate \cite{Hamid}. Hence we expect that there are at most $2$ embeddings of $\ON{PSL}_2(7)$ into the $\ON{Cr}_3$ with a del Pezzo fibration as one of the a regular models.

Okada has proven that a very general $X_n$ is not rational for $n\geqslant 5$ using the reduction to characteristic $2$ \cite{Okada15}. Applying Theorem 1.1 to varieties $X_n$ we get a better result.

\begin{Cor}
The varieties $X_n$ are birationally rigid for $n\geqslant 3$, in particular they are not rational.
\end{Cor}

The structure of the paper is the following. In section two we recall background knowledge on singularities and multiplicities of cycles. In section $3$ we prove the results related to the Cremona group: Theorem 1.7, Theorem 1.9, and Corollary 1.11. In section $4$ we recall the proof of rigidity of smooth del Pezzo fibrations of degree $2$. Sections $5$-$7$ are devoted to the proof of Theorem 1.1.

\subsection*{Acknowledgments}
The author would like to thank Ivan Cheltsov and Hamid Ahmadinezhad for many helpful conversations and suggestions, and Milena Hering for useful advice. Author started working on this problem during his visit to IBS and would like to thank Jihun Park for hospitality. The author was supported by the Edinburgh PCDS scholarship.

%
%Preliminaries
%

\section{Preliminaries}
We write $\equiv$ for numerical equivalence of $\MQ$-divisors and cycles and $\sim$ for linear equivalence of $\MQ$-divisors. We denote the symmetric group by $\mathcal{S}_n$ and its subgroup of even permutations by $\mathcal{A}_n$. All varieties are algebraic, normal and defined over $\MC$ unless stated otherwise. 

Let $X$ be an algebraic variety, possibly non-projective and singular. Let $E$ be a prime divisor on $X$. Then there is a discrete valuation $\nu_E$ of $\MC(X)$ corresponding to $E$ defined as $\nu_E(f)=\mult_E (f)$.

\begin{Def}[\cite{Valuations}]
Let $\varphi:\widetilde{X}\to X$ be a projective birational morphism. We say that a triple $(\widetilde{X},\varphi,E)$ is a \emph{realization} of a discrete valuation $\nu$ if $E$ is a prime divisor on $\widetilde{X}$ and $\nu_E=\nu$. We say that $\varphi(E)$ is the \emph{center} of a valuation $\nu_E$ on $X$.
\end{Def}

Note that if $X$ is projective, then every discrete valuation of the field $\MC(X)$ has a center on $X$ which does not depend on a realization.

\begin{Def}
Let $D$ be a divisor on $X$. We define the \emph{multiplicity} of a valuation $\nu$ at $D$ by the number 
\begin{align*}
\nu(D)=\mult_E \varphi^*(D)
\end{align*}
for some realization $(\widetilde{X},\varphi,E)$ of $\nu$. That is we can write
\begin{align*}
\varphi^*(D)=\varphi^{-1}(D)+\nu(D)E+\sum a_i E_i,
\end{align*}
where $E_i$ are the other exceptional divisors of $\varphi$. Multiplicity does not depend on the realization.
\end{Def}

\begin{Def}[{\cite[p.~6]{Pr_C}}]
Let $D$ be a $\mathds{Q}$-divisor on $X$ such that $K_X+D$ is $\mathds{Q}$-Cartier. Let $\pi:\widetilde{X}\to X$ be a birational morphism and let $\widetilde{D}=\pi^{-1}(D)$ be the proper transform of $D$. Then
\begin{align*}
K_{\widetilde{X}}+\widetilde{D}\sim\pi^*(K_X+D)+\sum_E a(E,X,D)E,
\end{align*}
where $E$ runs through all the distinct exceptional divisors of $\pi$ on $\widetilde{X}$ and $a(E,X,D)$ is a rational number. The number $a(E,X,D)$ \big(=$a(\nu_E,X,D)$\big) is called the \emph{discrepancy} of a divisor $E$ (valuation $\nu_E$) with respect to the pair $(X,D)$.
Let $\mathcal{M}$ be a linear system, not necessarily mobile, on $X$. Then we set $a(E,X,\mathcal{M})=a(E,X,D)$ for a generic divisor $D\in\mathcal{M}$.
\end{Def}

\begin{Def}[{\cite[p.~6]{Pr_C}}]
Let $\mathcal{M}$ be a linear system, not necessarily mobile, on $X$. We say that the pair $(X,\mathcal{M})$ is \emph{terminal} (resp. \emph{canonical}, \emph{log terminal}, \emph{log canonical}) \emph{at the valuation} $\nu$ with a center on $X$ if $a(E,X,\mathcal{M})>0$ (resp. $a(E,X,\mathcal{M})\geqslant0$, $a(E,X,\mathcal{M})>-1$, $a(E,X,\mathcal{M})\geqslant-1$) for some realization $(\widetilde{X},\varphi,E)$ of $\nu$. We say that the pair is $(X,\mathcal{M})$ terminal (resp. \dots) \emph{at a subvariety} $Z$ if it is \emph{terminal} (resp. \dots) at every valuation $\nu$ on $K(X)$ such that a center of $\nu$ on $X$ is $Z$. We say that the pair is $(X,\mathcal{M})$ \emph{terminal} (resp. \dots) if it is \emph{terminal} (resp. \dots) at every valuation with a center on $X$. If $\mathcal{M}=0$, we simply say that $X$ has only \emph{terminal} (resp. \emph{canonical}, \emph{log terminal}, \emph{log canonical}) singularities.
\end{Def}

\begin{Remark}
Consider the pair $(X,\mathcal{M})$. Let $f:Y\to X$ be a projective birational morphism, let $E_i$ be the exceptional divisors and let $\widetilde{\mathcal{M}}$ be the proper transform of $\mathcal{M}$ on $Y$. Then the pair
\begin{align*}
\big(Y,\widetilde{\mathcal{M}}-\sum a(E_i,X,\mathcal{M}) E_i\big)
\end{align*}
is called the \emph{log pullback} of the pair $(X,\mathcal{M})$.
It follows from the definition that the log pullback of the pair has the same singularities as the pair.
\end{Remark}

\begin{Lemma}[{\cite[Theorem~1.6]{Cheltsov14}}]
Let $\mathcal{M}$ be a mobile linear system on $\MC^2$. Let $C$ be a curve passing through the origin. Suppose the pair $(\MC^2,\frac{1}{n}\mathcal{M}-\alpha C)$ is not terminal at $0$ then
\begin{align*}
\mult_0 D_1\cdot D_2\geqslant 4n^2\alpha.
\end{align*}
\end{Lemma}

\begin{Prop}[{Corti inequality, \cite[Theorem~3.12]{Corti00}}]
Let $F_1$,\dots,$F_n \in \MC^3$ be irreducible surfaces passing through the origin. Let $\mathcal{M}$ be a mobile linear system on $\MC^3$ and let $Z=D_1\cdot D_2$ be the intersection of general divisors $D_1,D_2\in \mathcal{M}$. Write
$Z=Z_{h}+\sum Z_i$, where the support of $Z_i$ is contained in $F_i$ and $Z_h$ intersects $\sum F_i$ properly. 
Let $\alpha_i\geqslant0$ be rational numbers such that the pair $(\MC^3,\frac{1}{n}\mathcal{M}-\sum \alpha_i F_i)$ is not terminal at $0$. Then there are rational numbers $0 < t_i \leqslant 1$ such that
\begin{align*}
\mult_0 Z_h + \sum t_i \mult_0 Z_i\geqslant 4n^2(1+\sum \alpha_i t_i\mult_0 F_i).
\end{align*}
\end{Prop}

Note that decomposition $Z=Z_h+\sum Z_i$ may not be unique, but the inequality holds for any choice of the decomposition. Also note that we do not care if $\sum F_i$ is a normal crossing divisor, or if surfaces $F_i$ are smooth or not.

\begin{Lemma}[{\cite[Lemma~2.2.14]{IPFano}}]
Let $g:Y\to X$ be the blow up of a smooth curve $C\subset X$ on a smooth threefold $X$. Let $E$ be the exceptional divisor of $g$, then $E$ is a projectivization of the normal bundle $N_{C/X}$. Let $f\in A^2(Y)$ be the class of the fiber of the ruled surface $E$, then the following equalities hold
\begin{enumerate}[(i)]
	\item $E^2=-g^*(C)+\deg(N_{C/X})f$,
	\item $E^3=-\deg(N_{C/X})$,
	\item $E\cdot f=-1$,
	\item $E\cdot g^*(D)=(C\cdot D)f$,
	\item $f\cdot g^*(D)=0$,
	\item $E\cdot g^*(Z)=f\cdot g^*(Z)=0$,
	\item $\deg(N_{C/X})=2g(C)-2-K_X\cdot C$.
\end{enumerate}
\end{Lemma}

Let $Y$ be the quotient $\MA^n/\langle -I_n \rangle$. Let $P$ be the image of zero on $Y$, then we say that $P$ is a singular point of the type $\frac{1}{2}(1,\dots,1)$, or simply that $P$ is $\frac{1}{2}(1,\dots,1)$. If a singularity $Q$ is analytically isomorphic to $P$ we also say that $Q$ is $\frac{1}{2}(1,\dots,1)$. We call the singularities of the type $\frac{1}{2}(1,\dots,1)$ \emph{half-points}. Half-point is the simplest terminal quotient singularity.

\begin{Lemma}[{\cite[Lemma~4.10]{Pr_S}}]
Let $Q\in X$ be a $\frac{1}{2}(1,1,1)$-point. Suppose $f:Y\to X$ is the blow up of $Q$ and let $E$ be the exceptional divisor, then
\begin{enumerate}[(i)]
	\item $K_Y\sim f^*(K_X)+\frac{1}{2}E$,
	\item if the local equation of $D$ at $Q$ is $x_i=0$, then $f^{-1}(D)=f^*(D)-\frac{1}{2}E$,
	\item $\mathcal{O}_E(E)\vert_E=\mathcal{O}_E(-2)$.
\end{enumerate}
\end{Lemma}

\begin{Prop}[\cite{Kawamata}]
Let $f:\widetilde{X}\to X$ be the blow up of a half-point $Q$ and let $E$ be the exceptional divisor of $f$. Then a pair $(X,D)$ is canonical at $Q$ if and only if it is canonical at $E$, that is $a(E,X,D)<0$.
\end{Prop}

\begin{Cor}
Suppose a pair $(X,D)$ is canonical at a half-point $Q$. Then it is canonical at every curve passing through $Q$.
\end{Cor}
\begin{proof}
Suppose $(X,D)$ is not canonical at $C$ passing through $Q$, then $m=\mult_C D>1$. Let $f:\widetilde{X}\to X$ be the blow up of the point $Q$ and let $E$ be the exceptional divisor of $f$. Then $\nu_E(D)\geqslant\frac{m}{2}>\frac{1}{2}$. On the other hand $a(E,X,0)=\frac{1}{2}$ by Lemma 2.9, thus  $a(E,X,D)<0$ and the pair is not canonical at $Q$, contradiction.
\end{proof}

\begin{Lemma}
Let $Q\in X$ be a half-point. Suppose $D$ is a divisor passing through the point $Q$. Suppose also that there is a curve $C$ passing through the point $Q$ such that $C\cdot D=\frac{1}{2}$ and $C\not\subset \ON{Supp} D$. Then the pair $(X,D)$ is canonical at $Q$.
\end{Lemma}
\begin{proof}
Let $f:\widetilde{X}\to X$ be the blow up of the point $Q$. Let $\widetilde{C}$ and $\widetilde{D}$ be the proper transforms of $C$ and $D$ on $\widetilde{X}$ respectively. By the projection formula
\begin{align*}
0\leqslant \widetilde{D}\cdot \widetilde{C} = D\cdot C - \nu_E(D) E\cdot C=\frac{1}{2}-\nu_E(D),
\end{align*}
thus $\nu_E(D)\leqslant \frac{1}{2}$. By Lemma 2.9 and Proposition 2.10 the pair $(X,D)$ is canonical at $Q$.
\end{proof}

The following statements on the behavior of cycles on threefolds are well known. %, see \cite{Pukh123} for applications.
\begin{Lemma}
Let $Z$ be a $1$-cycle on a threefold $X$. Let $\sigma:\widetilde{X}\to X$ be the blow up of $B$ and let $E$ be the exceptional divisor. 
Then $\sigma^*{Z}=\sigma^{-1}Z+Z_E$, where $\ON{Supp} Z_E\subset E$ and
\begin{itemize}
	\item if $B$ is a nonsingular point, then $E\cong\MP^2$ and $\deg Z_E=\mult_B Z$.
	\item if $B$ is a smooth curve then $Z_E\equiv(C\cdot B)_S f$, where $f$ is a class of fiber of ruled surface $E$ and $S$ is some surface conatining $C$ and $B$ which is smooth at every point of $C\cap B$.
\end{itemize}
\end{Lemma}

\begin{Lemma}
Let $F$ be a hyperplane in $\MC^3$. Let $L$ be a curve in $F$ and let $C$ be an irreducible curve which does not lie in $F$. Let $\sigma: X\to \MC^3$ be a blow up of $L$. Let $E$ be the exceptional divisor and let $f\in A^{2}(X)$ be the class of a fiber of the ruled surface $E$. Then $\sigma^* C \equiv \sigma^{-1} C+kf$, where $k\leqslant C\cdot F$.
\end{Lemma}
\begin{proof}
By the projection formula and Lemma 2.8
\begin{align*}
k=E\cdot \sigma^* C - kE\cdot f=E\cdot \sigma^{-1} C\leqslant \sigma^* F \cdot \sigma^{-1}C=F\cdot C.
\end{align*}
\end{proof}

%Can be made shorter
\begin{Lemma}
Let $D_1$ and $D_2$ be generic divisors in a linear system $\mathcal{M}$ on $X$ and let $Z=D_1\cdot D_2$. Let $\sigma:\widetilde{X}\to X$ be the blow up of $B$, let $E$ be the exceptional divisor, and let $\widetilde{D}_i$ be the proper transform of $D_i$ on $\widetilde{X}$. Then
\begin{align*}
\widetilde{D}_1\cdot \widetilde{D}_2\equiv \sigma^{*}(Z)+Z_E,
\end{align*}
where $\ON{Supp} Z_E\subset E$. 

Suppose also that $B$ is a curve. Let $m=\nu_E(\mathcal{M})$ and let $f\in A^2(\widetilde{X})$ be a fiber of a ruled surface $E$. Then 
\begin{align*}
Z_E\equiv m^2E^2-2m(D_1\cdot B)f.
\end{align*}
\end{Lemma}
\begin{proof}
Since $\sigma^*D_i=\widetilde{D}_i+mE$, by Lemma 2.8
\begin{align*}
\widetilde{D}_1\cdot \widetilde{D}_2\equiv D_1\cdot D_2+m^2E^2-2m&\sigma^*(D_1\cdot E) 
\equiv \sigma^*(Z)+m^2E^2-2m(D_1\cdot B)f.
\end{align*}
\end{proof}

%
%section Cremona
%

\section{Klein simple group in Cremona group}
In this section we reprove Theorem 1.7 and prove Theorem 1.9.

\begin{Def}[\cite{GFanoI}]
Suppose a group $G$ acts on $X$. We say that $\pi:X\to Z$ is a $G$-\emph{Mori fiber space} if it satisfies the following conditions:
\begin{itemize}
	\item the variety $X$ is terminal and $G\mathds{Q}$-factorial, that is every $G$-invariant divisor on $X$ is $\MQ$-Cartier;
	\item morphism $\pi$ is flat, $G$-equivariant and, the invariant relative Picard rank $\rho^G(X/Z)=1$;
	\item the generic fiber of $\pi$ is a Fano variety.
\end{itemize}
\end{Def}

Lemma 1.4 implies that there is a rational $G$-Mori fiber space corresponding to every embedding of $G$ into $\ON{Cr}_n$. By a $G$-del Pezzo fibration $X$ we mean a three dimensional $G$-Mori fiber space over a projective line, in this case the general fiber is a $G$-del Pezzo surface. Varieties $X_n$ from the example 1.8 are the examples of $\ON{PSL}_2(7)$-del Pezzo fibrations. 

The group $\mathcal{A}_6$ has a unique central extension with a $3$-dimensional representation. Thus there is a unique action of $\mathcal{A}_6$ on $\MP^2$. The group $\ON{PSL}_2(7)$ has two $3$-dimensional representations. The representations are conjugate by the outer automorphism and the representations of central extensions induce the same action, hence the action of $\ON{PSL}_2(7)$ on $\MP^2$ is unique. There is also a unique del Pezzo surface $S_2$ of degree $2$ with the action of $\ON{PSL}_2(7)$, it is a double cover $r:S_2\to \MP^2$ branched over the Klein quartic. There are no other del Pezzo surfaces admitting the action of $\mathcal{A}_6$ or $\ON{PSL}_2(7)$.

\begin{Th}[{\cite[Theorem~1.4]{Belousov}}]
Let $S$ be a del Pezzo surface with log terminal singularities.
\begin{itemize}
	\item Suppose $S$ admits a $\ON{PSL}_2(7)$-action. Then $S$ is $\MP^2$ or $S_2$.
	\item Suppose $S$ admits an $\mathcal{A}_6$-action. Then $S$ is $\MP^2$.
\end{itemize}
\end{Th}

Thus a generic fiber of an $\mathcal{A}_6$-del Pezzo fibration is $\MP^2$ and of a $\ON{PSL}_2(7)$-del Pezzo fibration is $S_2$ or $\MP^2$. To prove Theorem 1.7 we show that an $\mathcal{A}_6$-del Pezzo fibration, can be transformed by an $\mathcal{A}_6$-equivariant fiberwise map into $\MP^2\times\MP^1$. Then we show that $\MP^2\times\MP^1$ does not have $\mathcal{A}_6$-equivariant fiberwise maps to $\mathcal{A}_6$-del Pezzo fibrations other than itself. We prove Theorem 1.9 similarly.

We say a map $\chi: X\dasharrow Y$ between fibrations $X\to B$ and $Y\to B$ is \emph{fiberwise} if it maps a fiber over $t\in B$ into a fiber over $t$ and it is an isomorphism on a generic fiber. Equivalently $\chi$ is a fiberwise map if it induces an isomorphism of the general fibers $X/B$ and $Y/B$ over $\MC(B)$.

\begin{Lemma}
Let $\pi: X\to B$ and $\pi^\prime: Y\to B$ be $G$-del Pezzo fibrations such that $G$ acts trivially on the base. Suppose that the general fibers $X/B$ and $Y/B$ of $\pi$ and $\pi^\prime$ are isomorphic as surfaces over ${\MC(B)}$ and suppose that surface $X/B$ admits a unique $G$-action up to isomorphism. Then there exists a $G$-equivariant fiberwise map $X\dasharrow Y$.
\end{Lemma}
\begin{proof}
An isomorphism map $X/B \to Y/B$ induces the isomorphism of fields $\chi^*: \MC(X)\to \MC(Y)$. Thus birational map corresponding to $\chi^*$ is a fiberwise map by definition. The group $G$ acts trivially on $\MC(B)$, hence the $G$-action on $X$ and $Y$ induces the action on $X/B$ and $Y/B$. Since the $G$-action on $X/B$ is unique, we may choose the isomorphism $\chi_B$ in such a way that it is $G$-equivariant. Then $\chi^*$ and the corresponding fiberwise maps are $G$-equivariant as well.
\end{proof}

Given a del Pezzo fibration we can consider its general fiber as a quartic in $\MP(1,1,1,2)_{\MC(t)}$. Algebraic operations we do on the equation of the general fiber correspond to the fiberwise transformations of the del Pezzo fibration.

\begin{Example}
Consider, a double cover $X$ of $\MC^1\times\MP^2$ branched over a central fiber $\{t=0\}$ and a divisor which is a Klein quartic in every fiber. Its general fiber is defined by the equation $w^2=t(x^3y+y^3z+z^3x)$ in $\MP_{\MC(t)}(1,1,1,2)$. 

Clearly, $X$ is canonical along preimage of a Klein quartic in the central fiber $F$. We may blow up this curve and then we can contract the proper transform of $F$ into a singular point of the type $\frac{1}{2}(1,1,1)$. Let $\widetilde{X}$ be the variety we acquire after performing these operations. The equation of its general fiber is $t(w^\prime)^2=x^3y+y^3z+z^3x$. Here we have made a coordinate change $wt=w^\prime$, and have divided both sides of the equation by $t$.
\end{Example}

\begin{Lemma}
Let $\pi:X\to \MP^1$ be a $\ON{PSL}_2(7)$-del Pezzo fibration of degree $2$. Then there is a fiberwise $\ON{PSL}_2(7)$-birational map to $X_n$ for some $n$.
\end{Lemma}
\begin{proof}
The general fiber $X/B$ of $\pi$ is a del Pezzo surface of degree $2$ over $\MC(t)$, that is it is a double cover of $\MP^2_{\MC(t)}$ branched over a quartic $q_4\in\MC[x,y,z](t)$. We can arrange the terms by the powers of $t$ 
\begin{align*}
q_4(x,y,z)=\sum_{n=-\infty}^{+\infty} p_i(x,y,z)t^n.
\end{align*}
Since $X$ and hence $q_4$ are $\ON{PSL}_2(7)$-invariant, all quartics $p_i$ are also $\ON{PSL}_2(7)$-invariant, therefore each $p_i$ is a multiple of Klein quartic. Thus $q_4(x,y,z)=s(t)(x^3y+y^3z+z^3x)$, for some $s\in \MC(t)$ and $X/B$ is a quartic in $\MP_{\MC(t)}(1,1,1,2)$ defined by the equation 
\begin{align*}
w^2=s(t)(x^3y+y^3z+z^3x).
\end{align*}

Let $s=\frac{q}{r}$, where $q,r\in\MC[t]$, then $\frac{r}{q}w^2=x^3y+y^3z+z^3x$. Let us change the coordinate $w=q\bar{w}$, then $r(t)q(t)\bar{w}^2=x^3y+y^3z+z^3x$. We may also assume that $rq$ does not have multiple roots, since we can change $\bar{w}$ again to get rid of them. Let $n$ be an integer such that $\deg rq=2n$ or $\deg(rq)=2n-1$. Consider a variety $X_n$ defined by the equation 
\begin{align*}
v^{2n}r\Big(\frac{u}{v}\Big)q\Big(\frac{u}{v}\Big)\bar{w}^2=x^3y+y^3z+z^3x.
\end{align*}
Since the generic fibers have the same equations there is fiberwise $\ON{PSL}_2(7)$-equivariant map from $X$ to $X_n$ by Lemma 3.3.
\end{proof}

We generalize \cite[Theorem~1.5]{CC} for our purposes.
\begin{Th}
Let $\pi:V\to Z$ be a $G$-Fano fibration over a curve $Z$. Suppose the pair $(V,\frac{1}{n}\mathcal{M})$ is canonical at any subset of a fiber $F$ for any $G$-equivariant mobile linear system $\mathcal{M}\subset \big| -nK_V + lF \big|$. Suppose $\varphi:V\dasharrow \bar{V}$ is a $G$-equivariant birational map to a $G$-Fano fibration $\bar{\pi}:\bar{V}\to Z$ such that
\begin{align*}
\xymatrix
{ 
	V\ar@{-->}[r]^\varphi \ar[d]_\pi & \ar[d]^{\bar{\pi}} \bar{V} \\
	Z\ar@{=}[r]^g & Z
}
\end{align*}
is commutative. Suppose $\varphi$ is an isomorphism outside $F$, then $\varphi$ is an isomorphism.
\end{Th}
\begin{proof}
Let $D$ be a very ample $G$-invariant divisor on $Z$ such that $-K_{V}+\pi^{*}(D)$ and $-K_{\bar{V}}+\bar{\pi}^{*}(D)$ are ample. %why need generic?
Put
$$
\Lambda=\big|-nK_{V}+\pi^{*}(nD)\big|,\ \Gamma=\big|-nK_{\bar{V}}+\bar{\pi}^{*}(nD)\big|,\ \bar{\Lambda}=\varphi(\Lambda),\ \bar{\Gamma}=\varphi^{-1}(\Gamma),%
$$
where $n$ is a natural number such that $\Lambda$ and $\Gamma$ have no base points.
Put
$$
M_{V}={\frac{2\varepsilon}{n}}\,\Lambda+{\frac{1-\varepsilon}{n}}\,\overline{\Gamma},\ M_{\bar{V}}={\frac{2\varepsilon}{n}}\,\bar{\Lambda}+{\frac{1-\varepsilon}{n}}\,\Gamma,%
$$
where $\varepsilon$ is a positive rational number. Note that these systems are $G$-invariant and for $\varepsilon$ small enough the linear systems $\big| K_{V}+M_{V} \big|$ and $\big|K_{\bar{V}}+M_{\bar{V}}\big|$ are ample. If the singularities of both log pairs $(V,M_{V})$ and $(V, M_{\bar{V}})$ are canonical, then $\varphi$ is an isomorphism by the uniqueness of canonical model \cite[Theorem~1.3.20]{Cheltsov-Fano}.

The linear system $\Gamma$ does not have base points, therefore for small enough $\varepsilon$ the pair $(\bar{V},M_{\bar{V}})$ is canonical. By the assumption of the theorem the pair $(V,{\frac{1}{n}}\bar{\Gamma})$ is canonical, thus $(V,{\frac{1-\varepsilon}{n}}\bar{\Gamma})$ is canonical. Therefore $(V,M_{V})$ is also canonical since $\Lambda$ has no base points.
\end{proof}

\begin{Lemma}[{\cite[Lemma~4]{Popov}}]
Suppose $G$ acts freely on a variety $X$ and $P$ is a $G$-invariant point. Then $G$ acts faithfully on $T_P X$.
\end{Lemma}

\begin{Cor}
The following assertions hold.
\begin{enumerate}[(i)]
	\item Suppose $\mathcal{A}_6$ acts on a surface $S$ and let $\Sigma$ be an orbit of a nonsingular point, then $\big| \Sigma \big|\geqslant 10$.
	\item Let $\Sigma$ be a $\ON{PSL}_2(7)$-orbit on $\MP^2$, then $\big| \Sigma \big|\geqslant 14$.
\end{enumerate}
\end{Cor} 
\begin{proof}
Let $H$ be the stabilizer of $P\in \Sigma$. Then by Lemma 3.7 there is an induced irreducible representation of $H$ on $T_P S\cong\MC^2$. The subgroups of $\mathcal{A}_6$ of index $<10$ are isomorphic to $\mathcal{A}_5$ \cite{GAP4} but $\mathcal{A}_5$ does not have irreducible $2$-dimensional representations, thus $(i)$ holds. 

The subgroups of $\ON{PSL}_2(7)$ of index $<14$ are $F_{21}$ and $\mathcal{S}_4$ (\cite{GAP4}). The group $F_{21}$ does not have irreducible $2$-dimensional representations. On the other hand, the induced action of $S_4$ on $\MP^2$ is irreducible (\cite{GAP4}). Thus $S_4$ cannot be a stabilizer and $(ii)$ holds.
\end{proof}

\begin{Lemma}
Let $r$ be a rational number.
\begin{enumerate}[(i)]
	\item Suppose $C\in \big| -rK_{\MP^2} \big|$ is an $\mathcal{A}_6$-invariant curve on $\MP^2$, then $r\geqslant 2$.
	\item Suppose $C\in \big| -rK_{\MP^2} \big|$ is a $\ON{PSL}_2(7)$-invariant curve on $\MP^2$, then $r\geqslant \frac{4}{3}$.
	\item Suppose $C\in \big| -rK_{S_2} \big|$ is a $\ON{PSL}_2(7)$-invariant curve on $S_2$, then $r\geqslant 2$.
	\item Let $S\in\MP(1_x,1_y,1_z,2_w)$ be a surface given by the equation $x^3y+y^3z+z^3x=0$. Suppose $C\in \big| -rK_{S} \big|$ is a $\ON{PSL}_2(7)$-invariant curve on $S$, then $r\geqslant 2$.
\end{enumerate}
\end{Lemma}
\begin{proof}
The action of $G$ on $\MP^2$ is induced by a $3$-dimensional representation of a central extension $\overline{G}$ of $G$. This representation induces representation of $\overline{G}$ on polynomials of degree $k$ as $\ON{Sym}^k(\MC^3)$. Therefore every $G$-invariant curve of degree $k$ corresponds to a $1$-dimensional $\overline{G}$-invariant subspace of $\ON{Sym}^k (\MC^3)$. The minimal $k$ such that $\ON{Sym}^k(\MC^3)$ has $1$-dimensional $\overline{G}$-invariant representations is $6$ for $\mathcal{A}_6$ and $4$ for $\ON{PSL}_2(7)$ \cite{GAP4}, this proves $(i)$ and $(ii)$

Suppose $C\in \big| -rK_{S_2} \big|$ is a  $\ON{PSL}_2(7)$-invariant curve on $S_2$. Without loss of generality we may assume that $C$ is $\ON{PSL}_2(7)$-irreducible. Consider the double cover $\pi: S_2\to \MP^2$ branched over Klein quartic $C_4$. 
Since $\pi$ is a canonical morphism, it is $\ON{PSL}_2(7)$-equivariant. Invariant Picard group $\ON{Pic}^G(S)=\MZ$, therefore $C$ is a pullback of a curve from $\MP^2$ or $C$ is a ramification divisor. In the first case $r\geqslant 4$, since $-K_{S_2}$ is a pullback of a line and the $\ON{PSL}_2(7)$-invariant curve of the lowest degree is a quartic. Clearly, in the latter case $r=2$.

There is a $G$-invariant curve $C_0\in \big| -2K_{S} \big|$ given by the equation $w=0$, note that $C_0$ does not pass through a singular point and $C_0$ is isomorphic to the Klein quartic. Suppose $C\neq C_0$, then $C\cap C_0$ is a union of orbits on $C_0$. Let $\Sigma\subset C\cap C_0$ be an orbit and let $H\subset \ON{PSL}_2(7)$ be a stabilizer of $\Sigma$. Then by Lemma 3.7 the group $H$ is cyclic. The cyclic subgroup of $\ON{PSL}_2(7)$ of maximal size has $7$ elements \cite{GAP4}, therefore $\big| \Sigma \big|\geqslant 24$. On the other hand $\big| C\cap C_0 \big|\leqslant C\cdot C_0=4r$, thus $4r\geqslant 24$. Hence $r\geqslant 6$ unless $C=C_0$, in which case $r=2$.
\end{proof}

\begin{Lemma}
Let $\mathcal{M}\subset \big|-rK_{\MP^2} \big|$ be a G-invariant linear system on $\MP^2$. Suppose $G$ is $\mathcal{A}_6$ or $\ON{PSL}_2(7)$ then the pair $(\MP^2,\frac{1}{r}\mathcal{M})$ is log canonical.
\end{Lemma}
\begin{proof}
Let $\mathcal{M}=F+\mathcal{M}^\prime$, where $F\in \big|-r_1K_{\MP^2} \big| $ is a fixed part and $\mathcal{M}^\prime \subset \big|-r_2K_{\MP^2} \big|$ is a mobile linear system. If both $(X,\frac{1}{r_1}F)$ and $(X,\frac{1}{r_2}\mathcal{M}^\prime)$ are log canonical, then $(\MP^2,\frac{1}{r}\mathcal{M})$ is log canonical.

Suppose the pair $(\MP^2,\frac{1}{r_2}\mathcal{M}^\prime)$ is not log canonical. It is log canonical at curves since it is mobile, hence it is not log canonical at some point $P$. Then it is not log canonical at the orbit $\Sigma$ of $P$. Consider generic divisors $D_1,D_2\in \mathcal{M}^\prime$, clearly $\mult_\Sigma D_i> r_2$. By Corollary 3.8
\begin{align*}
10r_2^2 \leqslant r_2^2 \big| \Sigma \big| < D_1\cdot D_2 = 9r_2^2,
\end{align*}
contradiction.

The pair $(X,\frac{1}{r_1}F)$ is log canonical by Lemma 3.9. 
\end{proof}

\begin{Lemma}
Let $S$ be a del Pezzo surface of degree $2$ with a $\ON{PSL}_2(7)$-action. Let $\mathcal{M}\subset \big| -rK_{S} \big|$ be a $\ON{PSL}_2(7)$-invariant linear system on $S$. Then the pair $(\MP^2,\frac{1}{r}\mathcal{M})$ is log canonical.
\end{Lemma}
\begin{proof}
The proof is analogous to the proof of Lemma 3.10.
\end{proof}

\begin{Lemma}
Let $S\subset\MP(1_x,1_y,1_z,2_w)$ be a surface given by the equation $x^3y+y^3z+z^3x=0$. Let $\mathcal{M}\subset \big |-rK_{S} \big|$ be a $\ON{PSL}_2(7)$-invariant linear system on $S$. Then the pair $(\MP^2,\frac{1}{r}\mathcal{M})$ is log canonical outside of the singular point of $S$.
\end{Lemma}
\begin{proof}
The pair is log canonical at curves by Lemma 3.9. Suppose it is not log canonical at a point $P$. Let $L$ be the curve from the ruling of the cone passing through the point $P$, then $L\cdot (-K_S)=\frac{1}{2}$. There is a number $k\geqslant0$ such that $\mathcal{M}\vert_S=kL+\mathcal{M}^\prime$, where $\mathcal{M}^\prime$ does not contain $L$. Since the pair is log canonical at curves $k\leqslant r$. Consider a generic divisor $D\in \mathcal{M}^\prime$. The pair $(S,\frac{1}{r}D+\frac{k}{r}L)$ is not log canonical at $P$, therefore by \cite[Theorem~7]{Tigers} 
\begin{align*}
r<(D\cdot L)_P\leqslant D\cdot L=r-\frac{k}{2}.
\end{align*}
\end{proof}

\begin{Prop}
Suppose $\pi:X\to \MP^1$ is a $G$-del Pezzo fibration. Suppose $f$ is a $G$-equivariant fiberwise birational map to a $G$-del Pezzo fibration. 
\begin{enumerate}[(i)]
	\item If $G$ is $\mathcal{A}_6$ and $X$ is $\MP^2\times\MP^1$ then $f$ is an isomorphism.
	
	\item If $G$ is $\ON{PSL}_2(7)$ and $X$ is $\MP^2\times\MP^1$ or $X_n$, then $f$ is an isomorphism.
\end{enumerate}
\end{Prop}
\begin{proof}
Suppose the map $f$ is not an isomorphism. Then by Theorem 3.6 there is a $G$-invariant linear system $\mathcal{M}\subset \big| -nK_X+lF \big|$ such that the pair $(X,\frac{1}{n}\mathcal{M})$ is not canonical at $B\subset F$. Lemma 2.12 implies that $B$ cannot be a singular point of the type $\frac{1}{2}(1,1,1)$ since curves from the ruling of the cone $F$ intersect divisors in $\mathcal{M}$ by $\frac{n}{2}$. Thus $B$ is a nonsingular point of $X$ or a curve. Then by Inversion of Adjunction (\cite[Theorem~17.7]{Kollar}) the pair $(F,\mathcal{M}\vert_F)$ is not log canonical at $B$ which contradicts Lemma 3.10, Lemma 3.11, or Lemma 3.12.
\end{proof}

\begin{proof}[Proof of Theorem 1.7]
Suppose $\pi:X\to \MP^1$ is an $\mathcal{A}_6$-del Pezzo fibration, then the general fiber $X/\MP^1$ of $\pi$ is $\MP^2_{\MC(t)}$ by Theorem 3.2. By Lemma 3.3 there exists an $\mathcal{A}_6$-equivariant fiberwise map from $X$ to $\MP^2\times\MP^1$. By Proposition 3.13 this map must be an isomorphism, thus $X\cong \MP^2\times\MP^1$.
\end{proof}
\noindent
Proof of Theorem 1.9 is analogous.

\begin{Lemma}
Varieties $X_n$ satisfy $K^2$-condition, that is $K_{X_n}^2\not\in\overline{NE}^\circ(X_n)$, if and only if $n\geqslant 3$.
\end{Lemma}
\begin{proof}
Let $E_Y$ be a unique effective divisor of bidegree $(2,-n)$, that is $E_Y$ is given by the equation $w=0$ on $Y_n$. Let $F_Y$ be a divisor of bidegree $(0,1)$, then $F_Y$ is a fiber of a $\MP(1,1,1,2)$-fibration $\pi_Y:Y_n\to \MP_1$. Let $S_Y$ be a surface of bidegree $(1,0)$, that is $S$ is given by an equation $x=0$. Clearly $E=E_Y\vert_{X_n}$ and $F=F_Y\vert_{X_n}$ generate the cone of effective divisors on $X_n$. Set $S_Y\vert_{X_n}=S$. Let $f$ be a line in a fiber and let $s=S\cdot E$, clearly $s$ is a bisection of $\pi=\pi_Y\vert_{X_n}$. It is easy to see that $s$ and $f$ generate the cone of effective curves of $X_n$.

Since $K_{Y_n}$ is a divisor of bidegree $(-5,n-2)$ by adjunction $K_{X_n}$ is a divisor of bidegree $(-1,n-2)$, thus 
\begin{align*}
K_{X_n}\sim -S + (n-2)F.
\end{align*}
On the other hand, equivalence $S\sim \frac{1}{2}E+\frac{n}{2}F$ implies
\begin{align*}
K_{X_n}\sim -\frac{1}{2}E + (\frac{n}{2}-2)F.
\end{align*}
The variety $X_n$ is a del Pezzo fibration of degree $2$, hence $S\cdot F\equiv 2f$. Therefore we compute
\begin{align*}
K^2_{X_n}=\big( -\frac{1}{2}E + (\frac{n}{2}-2)F \big)\cdot \big( -S + (n-2)F \big)  \equiv \frac{s}{2}+(8-3n)f.
\end{align*}
Clearly, $K^2_{X_n}$ is not in the interior of the Mori cone if and only if $n\geqslant 3$.
\end{proof}

Thus Theorem 1.1 is applicable and Corollary 1.11 holds.

%
%Rigidity of del Pezzo fibrations
%

\section{Rigidity of del Pezzo fibrations}
In this section we recall the proof of rigidity of smooth del Pezzo fibrations of degree $2$ and see what is the difficulty in the singular case.

\begin{Def}[{\cite{Pukh_Dir_Prod}, Definition 1}]
Let $\pi:V\to \mathds{P}^1$ be a del Pezzo fibration. We say that $V$ is \emph{birationally rigid} if for any birational map $f:V\dasharrow V^\prime$ to a Mori fiber space $\pi^\prime:V^\prime\to S$ the base $S$ is $\mathds{P}^1$, there is a commutative diagram
\begin{displaymath}
\xymatrix
{ 
	V\ar@{-->}[r]^f \ar[d]^\pi & \ar[d]^{\pi^\prime} V^\prime \\
	\mathds{P}^1\ar@{=}[r] & \mathds{P}^1,
}
\end{displaymath}
 and the general fibers of $\pi$ and $\pi^\prime$ are isomorphic.
\end{Def}

\begin{Th}[{Noether-Fano~inequality,~\cite{Pukh123}}]
Suppose $\pi:X\to\mathds{P}^1$ is a del Pezzo fibration. Suppose it satisfies $K^2$-condition, that is $K^2_X\not\in \overline{NE}^\circ (X)$. Suppose $g:X\dasharrow Y$ is a birational map which is not a morphism, then there exists a linear system $\mathcal{M}$ and a positive rational number $\lambda$ such that $\lambda\mathcal{M}+K_X\sim \gamma F$ and the pair $(X,\lambda\mathcal{M})$ is not canonical.
\end{Th}

To prove birational rigidity we would like to show that there are no such systems. But there are fiberwise maps and hence there are systems with non canonical singularities(Theorem 3.6). The method of supermaximal singularities was developed in \cite{Pukh123} to deal with this difficulty.

%Let $\pi^\prime:V\to Y$ be a Mori fiber space. and let  Let $\mathcal{M}_V=\big| \bar{\pi}^* H \big|$, where $H$ is a very ample divisor on $Z$. Let $\mathcal{M}$ be the proper transform of $\mathcal{M}_V$ on $X$, clearly the linear system $\mathcal{M}$ is mobile. Then $\mathcal{M}\subset \big| -nK_X+\gamma n F \big|$ for some $n$ and $\gamma$.

\begin{Prop}[{\cite{Pukh123}}]
Let $\pi:X\to \MP^1$ be a del Pezzo fibration satisfying $K^2$-condition.
Let $g:X\to V$ be a birational map to a Mori fiber space $\bar{\pi}:V\to Z$. Suppose also that a map $g$ is not fiberwise if $\bar{\pi}: V\to Z$ is a del Pezzo fibration. Let $\mathcal{M}=f^{-1}\big| \bar{\pi}^* H \big|$, where $H$ is a very ample divisor on $Z$. Then there are numbers $n>0$, $\gamma\geqslant0$ such that $\mathcal{M}\subset \big| -nK_X+\gamma n F \big|$ and one of the following holds.
\begin{enumerate}[(i)]
	\item There is a valuation $\nu$ of the field $K(X)$ such that its center on $X$ is a curve and the pair
$\Big(X,\frac{1}{n}\mathcal{M}\Big)$ is not canonical at $\nu$;

	\item There are finitely many valuations $\nu_i$ of $K(X)$ such that 
	\begin{itemize}
		\item the centers $P_i$ of $\nu_i$ are the points which all lie in different fibers,
		\item the pair $\Big(X,\frac{1}{n}\mathcal{M}\Big)$ is not canonical at every $\nu_i$,
		
		\item and the following inequality holds
\begin{align*}
-\sum \frac{a(\nu_i,X,\frac{1}{n}\mathcal{M})}{\nu_i(F_i)} > \gamma.
\end{align*}
	\end{itemize}
\end{enumerate}
\end{Prop}

%Note that there may be a choice for valuations $\nu_i$. The natural way to choose them would be to pick the one with the maximal discrepancy in each fiber.

Now suppose that $X$ is a del Pezzo fibration of degree $2$ with only $\frac{1}{2}(1,1,1)$-singularities. Suppose also that every fiber containing singularity is given by the equation $q_4(x,y,z)=0$ in $\MP(1_x,1_y,1_z,2_w)$.

\begin{Prop}
Let $C$ be a curve on $X$ which is not a section of $\pi$ and let $Q$ be a half-point. Then pair $(X,\frac{1}{n}\mathcal{M})$ is canonical at $C$ and $Q$.
\end{Prop}
\begin{proof}
It has been shown in \cite{Pukh123} that the pair $(X,\frac{1}{n}\mathcal{M})$ is canonical at $C$ if it is not a section and does not lie in a fiber containing a half-point.

Let $F$ be a fiber containing the half-point $Q$, then $F$ is a quartic cone by assumption. The curves from the ruling of $F$ intersect $-K_X$ by $\frac{1}{2}$, therefore the pair $(X,\frac{1}{n}\mathcal{M})$ is canonical at half-points by Lemma 2.12. Thus by Corollary 2.11 the pair is also canonical at every curve passing through the half-point.

Suppose that a curve $C$ does not pass through $Q$. Let $L$ be a curve from the ruling of $F$ such that $L\not\subset \ON{Supp} C$. Then the intersection $C\cdot L\geqslant 1$ since it is a positive integer. On the other hand $C\equiv rL$ for some $r$ and 
\begin{align*}
C\cdot L=rL^2=\frac{r}{8}\geqslant 1.
\end{align*}
Suppose the pair is not canonical at $C$. Then for a generic $D\in \mathcal{M}$ there is the decomposition $D\vert_F=kC+D^\prime$, where $k>1$ and $D^\prime$ is an effective divisor on $F$ which does not contain $L$. It is impossible since 
\begin{align*}
D\vert_F\equiv -K_F\equiv 4L.
\end{align*}
\end{proof}

\begin{Prop}[{\cite[Section~3]{Pukh123}}]
Let $C$ be a section of $\pi$. Then either the pair $(X,\frac{1}{n}\mathcal{M})$ is canonical at $C$ or there exists a birational involution $\chi_C$ such that
\begin{itemize}
	\item the following diagram is commutative
\begin{displaymath}
\xymatrix
{ 
	X\ar@{-->}[r]^{\chi_C} \ar[d]^\pi & \ar[d]^{\pi} X \\
	\mathds{P}^1\ar@{=}[r] & \mathds{P}^1,
}
\end{displaymath}
	\item there are numbers $n^\prime<n$ and $\gamma^\prime$ such that $\chi_C(\mathcal{M})\subset \big| -n^\prime K_X +n^\prime\gamma^\prime \big|$, and
	\item the pair $\big(X,\frac{1}{n^\prime}\chi_C(\mathcal{M})\big)$ is canonical at $C$.
\end{itemize}
\end{Prop}

The maps $\chi_C$ are the kind which we allow in the definition of rigid varieties. We prove that every map to a Mori fiber space is a composition of $\chi_C$ and fiberwise maps. Using Proposition 4.5 we can \emph{untwist} the curve $C$, that is we replace the map $f$ with $f\circ \chi_C$ and gain less singular linear system. Thus we only have to deal with nonsingular points, they are the real difficulty. 

%Supermaximal singularity

\subsection{Supermaximal singularities}
Suppose we are in the case $(ii)$ of Proposition 4.3. By Proposition 4.4 we may assume that $X$ is smooth at the points, where the pair $(X,\frac{1}{n}\mathcal{M})$ is not canonical. Let $D_1,D_2\in\mathcal{M}$ be generic divisors and let $Z=D_1\cdot D_2$.
For some point $P$ we bound $\mult_P Z$ from above using the degrees and from below using Corti inequality. We show in this section that these bounds contradict if the fiber containing $P$ does not pass through half-points.

Let $C\subset X$ be the irreducible curve. We say that $C$ is \emph{horizontal} if $\pi(C)=\MP^1$ and that it is \emph{vertical} if $\pi(C)$ is a point. We say that a cycle $C$ is vertical (horizontal) if every curve in $C$ is vertical (horizontal). We decompose $Z$ into the vertical and the horizontal components
\begin{align*}
Z=\sum Z^v_t+Z^h,
\end{align*}  
where the support of $Z_t^v$ is in the fiber $F_t=\pi^{-1}(t)$, $t\in \MP^1$. Define the degree of a vertical $1$-cycle $C^v$ by the number $\deg C^v=C^v\cdot(-K_X)$ and degree of a horizontal $1$-cycle $C^h$ be the number $\deg C^h=C^h\cdot F$.

\begin{Lemma}[\cite{Pukh123}]
Let $C$ be an irreducible curve on $X$ and let $P\in X$ be a nonsingular point such that $P\in C$. Then
\begin{enumerate}[(i)]
	\item if $C$ is horizontal then $\mult_P C \leqslant \deg C$,
	\item if $C$ lies in a fiber $F$, which does not contain a half-point then $\mult_P C \leqslant 2\deg C$.
\end{enumerate}
\end{Lemma}
%\begin{proof}
%Suppose $C$ is horizontal. Then $\mult_P F\geqslant 1$ implies
%\begin{align*}
%\mult_P C\leqslant C\cdot F \leqslant -K_F\cdot C=\deg C.
%\end{align*}

%Suppose $C$ is a curve in a fiber. Consider the pencil of elliptic curves $\mathcal{L}_P\subset \big| -K_F \big|$ passing through the point $P$. Then for a generic $L\in\mathcal{L}_P$
%\begin{align*}
%\mult_P C\leqslant C\cdot L \leqslant -K_F\cdot C=2\deg C.
%\end{align*}
%\end{proof}

\begin{Lemma}[\cite{Pukh123}]
The following holds for the degrees of $Z^h$ and $Z^v$
\begin{align*}
&\deg Z^h=2n^2 \quad and \quad  \sum_{t\in\MP^1} \deg Z_t^v\leqslant 4n^2\gamma.
\end{align*}
\end{Lemma}
%\begin{proof}
%The class of the cycle $Z$ is 
%\begin{align*}
%(-nK_X+\gamma nF)^2\equiv n^2K_X^2+4n^2\gamma f.
%\end{align*}
%It follows from the $K^2$-condition and the effectiveness of the cycle $Z^h$ that the class of $Z^v$ is $l f$, where $l\leqslant 4n^2\gamma$. Let us intersect the cycle $Z$ with $F$ to find the degree of $Z^h$.
%\begin{align*}
%\deg Z^h=F\cdot (n^2K_X^2+4n\gamma f)=n^2(F\cdot K_X^2)=n^2\big(2f\cdot (-K_X)\big)=2n^2.
%\end{align*}
%\end{proof}

Now we need to find which pair and which point do we apply Corti inequality to. Let $F_i$ be the fibers containing the centers $P_i$ of $\nu_i$. Let $Z^v_i$ be the part of vertical cycle which is contained in $F_i$.

\begin{Lemma}
There are numbers $\gamma_i$ such that the pair $(X,\frac{1}{n}\mathcal{M}-\sum \gamma_i F_i)$ is strictly canonical at each $\nu_i$ and $\sum \gamma_i>\gamma$.
\end{Lemma}
\begin{proof}
Set $\gamma_i=-\frac{a(\nu_i, X, \frac{1}{n}\mathcal{M})}{\nu_i(F_i)}$, these numbers satisfy the inequality by the Proposition 4.3. The pair is obviously strictly canonical at every $\nu_i$.
\end{proof}

\begin{Cor}
There is an index $i$ such that 
\begin{align*}
\deg Z^v_i<4n^2\gamma_i.
\end{align*}
\end{Cor}
\begin{proof}
By Lemma 4.7 and Lemma 4.8 we have
\begin{align*}
\sum \deg Z^v_i\leqslant 4n^2\gamma <4n^2 \sum \gamma_i.
\end{align*}
If inequality holds for the sums, then it must hold for at least one $i$.
\end{proof}

We say that $\nu_i$ a \emph{supermaximal singularity} if $\deg Z_i^v<4n^2\gamma_i$. Fix a supermaximal singularity $\nu_i$. To simplify the notations, from now on denote $F_i$ as $F$, $\gamma_i$ as $\gamma$, $P_i$ as $P$, $Z^v_i$ as $Z_v$, $Z^h$ as $Z_h$ and $\nu_i$ as $\nu$. 

\begin{Prop}[\cite{Pukh123}]
Suppose a fiber $F$ does not contain a singular point of the type $\frac{1}{2}(1,1,1)$, then there are no supermaximal singularities with a center on $F$.
\end{Prop}
\begin{proof}
The pair $(X,\frac{1}{n}\mathcal{M}-\gamma F)$ is strictly canonical at the point $P$. Hence by Corti inequality there is a number $0 < t \leqslant 1$ such that 
\begin{align*}
\mult_P Z_h+t \mult Z_v\geqslant 4n^2(1+\gamma t \mult_P F)\geqslant 4(1+\gamma t)n^2.
\end{align*}
On the other hand Lemma 4.7 implies
\begin{align*}
\mult_P Z_h+t \mult Z_v<2n^2+4t\gamma n^2,
\end{align*}
contradiction.
\end{proof}

\begin{Cor}[\cite{Pukh123}]
Let $X$ be a smooth del Pezzo fibration of degree $2$. Suppose $X$ satisfies $K^2$-condition. Then $X$ is birationally rigid.
\end{Cor}
%\begin{proof}
%Suppose $X$ is not birationally rigid, that is there is a birational map $g$ to a Mori fiber space which does not satisfy requirements of Definition 5.1. Then by Proposition 5.3 there is a system $\mathcal{M}\subset \big| -nK_X+\gamma nF \big|$ such that one of the conditions $(i)$ and $(ii)$ holds. 

%Suppose the pair $(X,\frac{1}{n}\mathcal{M})$ is not canonical at a curve $C$. By Proposition 5.4 the curve $C$ is a section. Then by Proposition 5.5 there is a birational involution $\chi_C$ such that $g\circ\chi_C$ does not satisfy conditions of Definition 5.1. Now we consider the map $g\circ \chi_C$ instead of $g$ and the linear system 
%\begin{align*}
%\chi_C(\mathcal{M})\subset \big| -n^\prime K_X+n^\prime \gamma^\prime \big|
%\end{align*}
%instead of $\mathcal{M}$. We repeat this process until either $n^\prime=0$, that is the new map $g^\prime=g\circ \chi_{C_1}\circ\dots\circ\chi_{C_k}$ is fiberwise (contradiction), or until the pair, corresponding to $g^\prime$ is canonical at all curves.

%Now we are in the situation $(ii)$ of Proposition 5.3. Hence by Corollary 5.9 there is a supermaximal singularity, but it contradicts Proposition 5.10.
%\end{proof}

\begin{Remark}
Note that there are difficulties when $F$ contains the half-point. There will be ``half-line'': curves of degree $\frac{1}{2}$. Thus the bound on multiplicity becomes $\mult_P Z_v\leqslant 8\gamma n^2$ and it no longer contradicts Corti inequality. In the next sections we work on a way around this problem.
\end{Remark}

%
%Construction of the ladder
%
\section{Construction of the ladder}
Let $X^{(0)}$ be a threefold and let $F^{(0)}$ be a smooth surface on it. Suppose $L_0\subset F^{(0)}$ is a smooth rational curve. Let us associate the following construction to $L_0$ which we call the \emph{ladder}. 

Let $\sigma_i:X^{(i)}\to X^{(i-1)}$ be the blow up of $L_{i-1}$ and let $E^{(i)}$ be its exceptional divisor. Clearly $E^{(i)}\cong\mathds{F}_m$ for some $m$, suppose $m>0$ for every $i$. Let $L_i$ be the exceptional section of $E^{(i)}$. Denote the proper transform of $E^{(i)}$ on $X^{(j)}$, as $E^{(i,j)}$ and the proper transform of $F^{(0)}$ on $X^{(j)}$ as $F^{(j)}$.

\begin{Th}
Suppose $L_0$ is a smooth rational curve and suppose $X^{(0)}$ is smooth in the neighborhood of $L_0$. Also assume that $L_0\cdot K_{X^{(0)}}=0$ and $L_0 \cdot F^{(0)}=-2$. Then the following assertions are true for the ladder associated to $L_0$ for all $i>0$:
\begin{enumerate}[(i)]
	\item $E^{(i)}\cong \mathds{F}_2$,
	\item $\sigma_i^*(L_{i-1})\equiv L_{i}$,
	\item $\nu_{E^{(i)}}(F^{(0)})=i$,
	\item $E^{(i,i+1)}\vert_{E^{(i+1)}}$ is disjoint from $L_{i+1}$, in particular the graph associated to the ladder is a simple chain.
\end{enumerate}
\end{Th}

\begin{Lemma}
Let $\sigma:\widetilde{X}\to X$ be the blow up of a smooth rational curve $L$ and let $E$ be the exceptional divisor of $\sigma$. Suppose $X$ is smooth in the neighborhood of $L$ and suppose $L\cdot K_X=0$. Suppose also that there is a smooth surface $F$ such that $L\cdot F=-2$, then $E\cong \mathds{F}_2$.
\end{Lemma}
\begin{proof}
By Lemma 2.8
\begin{align}
\deg N_{L/X}=2g(L)-2-K_{X}\cdot L=-2.
\end{align}
The equality $F\cdot L=-2$ implies 
\begin{align}
N_{F/X}\vert_L=\mathcal{O}_L(-2).
\end{align}

There is an exact sequence of normal sheaves
\begin{align*}
0\to N_{L/F} \to N_{L/X} \to (N_{F/X})\big\vert_{L} \to 0.
\end{align*}
Clearly $N_{L/X}=\mathcal{O}_{L}(a)\oplus\mathcal{O}_{L}(b)$ for some $a$ and $b$, and $(i)$ implies that $a+b=-2$. Without loss of generality we may assume that $a\leqslant b$. The inequality $a\leqslant-2$ follows from $(2)$. On the other hand $(1)$, $(2)$, and the exact sequence imply $N_{L/F}=\mathcal{O}_L$, therefore $b\geqslant 0$. Hence $a=-2$ and $b=0$.
Thus $E=\ON{Proj} \big( \mathcal{O}_{L}(-2)\oplus\mathcal{O}_{L} \big)\cong \mathds{F}_2$.
\end{proof}

\begin{Lemma}
Let $f_i,s_i\in A^2(X^{(i)})$ be the classes of a fiber and of the exceptional section of a ruled surface $E^{(i)}$ respectively. Suppose that $E^{(i)}\cong \mathds{F}_2$, $K_{X^{(i-1)}}\cdot L_{i-1}=0$, and $F^{(i-1)}\cdot L_{i-1}=-2$.
Then:
\begin{enumerate}[(i)]
	\item $E^{(i+1)} \cong\mathds{F}_2$,
	\item $\sigma_i^*(L_{i-1})\equiv L_{i}$,
	\item $K_{X^{(i)}}\cdot L_{i}=0$,
	\item $F^{(i)}\cdot L_i=-2$.
\end{enumerate}
\end{Lemma}
\begin{proof}
Lemma 5.2 implies $(i)$.

By Lemma 2.8
\begin{align*}
0=E^{(i)}\cdot\sigma_i^*(L_{i-1})=E^{(i)}\vert_{E^{(i)}}\cdot\sigma_i^*(L_{i-1})=\big( \sigma_i^*(L_{i-1})-2f_i \big)\cdot\sigma_i^*(L_{i-1}).
\end{align*}
Thus $\sigma_i^*(L_{i-1})^2=-2f\cdot\sigma_i^*(L_{i-1})$.
Clearly $\sigma_i^*(L_{i-1})$ must be a section. Indeed, by Lemma 2.8
\begin{align*}
2=\big(E^{(i)}\big)^3=\Big( (\sigma_i^*(L_{i-1})-2f_i \Big)^2=2f_i\cdot\sigma_i^*(L_{i-1}),
\end{align*}
therefore $f_i\cdot\sigma_i^*(L_{i-1})=1$, that is $\sigma_i^*(L_{i-1})$ is a section. Since $\sigma_i^*(L_{i-1})^2=-2$, as computed above, it is the exceptional section $L_i$.

It follows from $(ii)$, that $L_i\cdot E^{(i)}=0$. Thus 
\begin{align*}
K_{X^{(i)}}\cdot L_{i}=K_{X^{(i-1)}}\cdot L_{i-1}+E^{(i)}\cdot L_i=0.
\end{align*}
Similarly
\begin{align*}
F^{(i)}\cdot L_i=F^{(i-1)}\cdot L_{i-1}+ E^{(i)}\cdot L_i=F^{(i-1)}\cdot L_{i-1}=-2.
\end{align*}
\end{proof}

%\begin{Remark}
%If we did the calculations only using the exceptional divisors we would only know that $E^{(i)}$ is the $\mathds{F}_2$ or the $\MP^1\times\MP^1$. If it is the latter surface at some point, then we get a flop of the curve $L^{(0)}$. But this curve moves in the family which is impossible.
%(Reid's pagoda)
%\end{Remark}

\begin{proof}[Proof~of~Theorem~5.1]
Lemma 5.2 and Lemma 5.3 imply $(i)$ and $(ii)$.

Clearly $\nu_{E^{(1)}}(F^{(0)})=1$. On the other hand $L_{M-1}\subset F^{(M-1)}$ since $L_{M-1}\cdot F^{(M-1)}<0$. Hence $\nu_{E^{(M)}}(F^{(0)})=\nu_{E^{(M-1)}}(F^{(0)})+1$ and $(iii)$ holds.

By Lemma 2.8 and $(i)$ 
\begin{align*}
E^{(i-1,i)}\vert_{E^{(i)}}=\big(L_{i-1}\cdot E^{(i-1)}\big)f-E^{(i)}\vert_{E^{(i)}}=s_i+2f_i.
\end{align*}
Therefore $E^{(i-1,i)}\vert_{E^{(i)}}\cdot L_i=0$ and $(iv)$ holds.
\end{proof}

Suppose $\pi:X\to \MP^1$ is a del Pezzo fibration of degree $2$. Suppose $Q\in X$ is a $\frac{1}{2}(1,1,1)$ point and $F$ is a fiber containing $Q$. Suppose that the fiber $F$ can be embedded into $\MP(1_x,1_y,1_z,2_w)$ as a cone $q_4(x,y,z)=0$. Let $\sigma_Q:X^{(0)}\to X$ be the blow up of $X$ at $Q$ and let $E_Q$ be the exceptional divisor of $\sigma_Q$. 

Let $L\subset F$ be a ``half-line'', that is a curve $L$ such that $L\cdot(-K_F)=\frac{1}{2}$. Note that every curve in a ruling of $F$ is a half-line and that every half-line is such. Denote the proper transforms of $L$ and $F$ on $X^{(0)}$ as $L_0$ and $F^{(0)}$ respectively . Since $F$ is a cone, its blow up $F^{(0)}$ is a ruled surface over a curve of genus $3$. Cleary the curve is a plane quartic $q_4(x,y,z)=0$. The curve is smooth, since $X$ has only $\frac{1}{2}(1,1,1)$-singularities. We can construct the ladder associated to $L_0$. We also say that the ladder is associated to the half-line $L$. Now we show that $X^{(0)}$, $F^{(0)}$, and $L_0$ satisfy the assumptions of the Theorem 5.1.

\begin{Lemma}%[Must have easier proof]
The following equalities hold
\begin{enumerate}[(i)]
	\item $L_0\cdot E_Q=1$,
	\item $L_0\cdot F^{(0)}=-2$,
	\item $K_{X^{(0)}}\cdot L_0=0$.
\end{enumerate}
\end{Lemma}
\begin{proof}
Since $L_0\subset F^{(0)}$ we have $L_0\cdot E_Q=L_0\cdot E_Q\vert_{F^{(0)}}$. Let $\MP=\MP(1,1,1,2)$, and consider the embedding of $F$ into $\MP$. We can describe $L\subset F$ in $\MP$ as an intersection $H_1\cdot H_2$, for some $H_i\in\big| \mathcal{O}_\MP(1) \big|$. Let $\sigma_\MP:\widetilde{\MP}\to\MP$ be the blow up of the point $Q$ and let $E_{\MP}$ be its exceptional divisor. Clearly $\sigma_\MP:\sigma_\MP^{-1}(F)\to F$ is the blow up of $Q$ thus without any confusion we may identify $\sigma_\MP^{-1}(F)$ with $F^{(0)}$ and $\sigma_\MP^{-1}(L)$ with $L_0$. Let $\widetilde{H}_i$ be the proper transform of $H_i$ on $\widetilde{\MP}$, then $L_0=\widetilde{H}_1\cdot \widetilde{H}_2$.
Denote the exceptional divisor of $\sigma_\MP$ as $E_\MP$, then
\begin{align*}
L_0\cdot E_Q\vert_{F^{(0)}}=L_0\cdot E_\MP\vert_{F^{(0)}}=L_0\cdot E_\MP=\widetilde{H}_1\cdot \widetilde{H}_2 \cdot E_\MP=\widetilde{H}_1\vert_{\widetilde{H}_2}\cdot E_\MP\vert_{\widetilde{H}_2}.
\end{align*}
Clearly $H_2$ is isomorphic to $\MP(1,1,2)$ and $\sigma_\MP\vert_{\widetilde{H}_2}$ is the blow up of a singular point. Thus $\widetilde{H}_2\cong \mathds{F}_2$, $E_\MP\vert_{\widetilde{H}_2}$ is the exceptional section, and $\widetilde{H}_1\vert_{\widetilde{H}_2}$ is a fiber of $\widetilde{H}_2$. Hence
\begin{align*}
1=\widetilde{H}_1\vert_{\widetilde{H}_2}\cdot E_\MP\vert_{\widetilde{H}_2}=L_0\cdot E_Q.
\end{align*}

Consider the affine open subset $U\in\MP$ given by the $w\neq0$. Clearly $U=\MC^3/\langle -I_3 \rangle$ and the local equation of $F$ at $Q$ on $U$ is $q_4(x,y,z)=0$. Thus Lemma 2.9 implies $F^{(0)}=\sigma^*_\MP(F)-2E_Q$ and we find the intersection
\begin{align*}
L_0\cdot F^{(0)}=-2L_0\cdot E_Q=-2.
\end{align*}
The equality $(iii)$ follows from $(i)$ and Lemma 2.9
\begin{align*}
K_{X^{(0)}}\cdot L_0=(\sigma_Q^*{K_X}+\frac{1}{2}E_Q)\cdot L_0=K_X\cdot L+\frac{1}{2}E_Q\cdot L_0=0.
\end{align*}
\end{proof}

%
%Multiplicities on the ladder
%

\section{Multiplicities on the ladder}
The plan is to associate a ladder to a half-line, to apply Corti inequality upstairs, and to derive a contradiction. Thus we need to find bounds on multiplicities of the cycles upstairs.

Let $A$ be a cycle, a divisor or a linear system on $X$. We denote its proper transform on $X^{(i)}$ as $A^{(i)}$. For divisors and cycles on $X^{(j)}$ we add upper index. For example, $E^{(1,3)}$ is the proper transform of $E^{(1)}$ on $X^{(3)}$. By $\sigma^*$ we mean the appropriate composition of $\sigma^*_i$. For example, $E^{(1,3)}=\sigma^*(E^{(1)})-E^{(2,3)}-E^{(3)}$, here $\sigma^*=\sigma^*_2\circ\sigma_1^*$. 

\begin{Prop}[\cite{Pukh123}]
Let $X^{(0)}$ be a threefold and let $F^{(0)}$ be a surface in it. Suppose $L_0$ is a smooth rational curve in $F^{(0)}$. Let $\sigma_i:X^{(i)}\to X^{(i-1)}$ be the associated ladder. Let $\nu$ be a discrete valuation of $K(X^{(0)})$ and suppose that a center of $\nu$ on $X^{(0)}$ is a point on $L_0$. Then there is a number $M$ such that for every $i<M$ the center of $\nu$ on $X^{(i)}$ is a point on the exceptional section $L_i$ and the center of $\nu$ on $X^{(M)}$ is
\begin{enumerate}[A)]
	\item a fiber of a ruled surface $E^{(M)}$,
	\item a point not on $L_M$ and not on $E^{(M-1,M)}$, or
	\item a point on $E^{(M)}\cap E^{(M-1,M)}$.
\end{enumerate}
\end{Prop}

Suppose the linear system $\mathcal{M}\subset\big|-nK_X+lF\big|$ has a supermaximal singularity $\nu$ at a nonsingular point $P\in X$. Let $F$ be a fiber containing $P$ and suppose $F$ contains half-point. Recall the notations of Subsection 4.1. Let $Z=D_1\cdot D_2$ for generic divisors $D_1,D_2\in\mathcal{M}$. Let $Z_h$ be the horizontal part of $Z$ and let $Z_v$ be the part of $Z$ which lies in $F$.
Let $\gamma$ be the number such that the pair 
%\begin{align*}
$(X,\frac{1}{n}\mathcal{M}-\gamma F)$
%\end{align*}
is strictly canonical at $\nu$.

Let $L$ be a unique curve from the ruling of $F$ passing through $P$. The cycle $Z_v$ can be decomposed as $Z_v=kL+C$, where $k\geqslant 0$ and $C$ does not contain $L$.

\begin{Lemma}
The inequality $C\cdot L\leqslant \gamma n^2$ holds.
\end{Lemma}
\begin{proof}
Note that $\Big( \ON{Div}F/\equiv \Big)=\MZ$. Thus it is easy to see that $-K_F\equiv 4L$ and that $C\equiv rL$ for some $r$. 
By Corollary 4.9 there is a bound $\deg (kL+C)\leqslant 4\gamma n^2$,
therefore $r\leqslant 8\gamma n^2-k$. Hence
\begin{align*}
C\cdot L=rL^2=\frac{r}{8}\leqslant \gamma n^2-\frac{k}{8}\leqslant \gamma n^2.
\end{align*}
\end{proof}

\begin{Lemma}
Let $\nu_Q$ be the valuation corresponding to the exceptional divisor $E_Q$ of a blow up of a half-point $Q$. Then for a generic $D\in\mathcal{M}$
\begin{align*}
D^{(i)}\cdot L_{i}=\frac{n}{2}-\nu_Q(D).
\end{align*}
\end{Lemma}
\begin{proof}
Since $\sigma_i^*L_{i-1}=L_i$ the equality $D^{(0)}\cdot L_{0}=D^{(i)}\cdot L_{i}$ holds for all $i$. Theorem 6.1 implies
\begin{align*}
D^{(0)}\cdot L_0=\sigma_Q^*(D)\cdot L_0 - \nu_Q(D)E_Q\cdot L_0=\sigma_Q^*(D)\cdot L_0-\nu_Q(D).
\end{align*}
On the other hand, by Lemma 2.8
\begin{align*}
\sigma_Q^*(D)\cdot L_0=D\cdot L=-nK_X\cdot L=\frac{n}{2}.
\end{align*}
Combining the equalities we get the statement of the lemma.
\end{proof}

Denote $Z_i=D^{(i)}_1\cdot D^{(i)}_2$, then by Lemma 2.13 there is the decomposition 
\begin{align*}
Z_0=Z_v^{(0)}+Z_h^{(0)}+Z_Q,
\end{align*}
where $Z_Q$ is the part of the cycle which lives on the exceptional divisor. We disregard the part $Z_Q^{(i)}$ in further computations since it is away from from the center of $\nu$.

For every $i>0$ there is a part $C_i$ of the cycle $Z^{(i)}$ which lives on $E^{(i)}$. Recall that $E^{(i)}$ is a ruled surface $\mathds{F}_2$ and $\sigma_i\vert_{E^{(i)}}$ is the $\MP^1$-fibration. We say that a curve $B$ on $E^{(i)}$ is vertical if $\sigma_i(B)$ is a point and horizontal otherwise. The cycle $C_i$ can be decomposed into the sum of the exceptional section with multiplicity, the rest of the horizontal part, and the vertical part: 
\begin{align*}
C_i=k_iL_i+C^{(i)}_h+C^{(i)}_v.
\end{align*}
Note that $\sigma^* C^{(i-1,i)}=C^{(i-1,i+k)}$ for any $k>0$, $i>1$ since $E^{(i-1,i)}$ is disjoint from $L_i$. Thus there are the decompositions
\begin{align*}
&Z_0=Z_h^{(0)}+Z_v^{(0)}=Z_h^{(0)}+C^{(0)}+k_0L_0,\\
&Z_1=Z_h^{(1)}+C^{(1)}+C_h^{(1)}+C_v^{(1)}+k_1L_1,\\
&Z_2=Z_h^{(2)}+C^{(2)}+C_h^{(1,2)}+C_v^{(1,2)}+C_h^{(2)}+C_v^{(2)}+k_2L_2,\\
&Z_i=Z_h^{(i)}+C^{(i)}+\sigma^*C_h^{(1,2)}+\sigma^*C_v^{(1,2)}+\dots+C_h^{(i-1,i)}+C_v^{(i-1,i)}+C_h^{(i)}+C_v^{(i)}+k_iL_i.
\end{align*}

Let $\lambda_i=\mult_{L_{i-1}}\mathcal{M}^{({i-1})}$ and recall that $f_i, s_i\in A^2(X^{(i)})$ are the classes of a fiber and of the exceptional section of $E^{(i)}$ respectively.
Thus $C_v^{(i)} \equiv d_v^{(i)}f_i$ and $C_h^{(i)}\equiv d_h^{(i)}s_i+\beta_if_i$ for some $d_v^{(i)}$, $d_h^{(i)}$, and $\beta_i$. Also $2d_h^{(i)}\leqslant \beta_i$ because $C_h^{(i)}$ does not contain the exceptional section.

\begin{Lemma}
The following relations for the proper transforms and the pullbacks of the cycles hold
\begin{align*}
&C_h^{(i,i+1)} \equiv d^{(i)}_h s^i + \beta_i f_i - (\beta_i - 2d^{(i)}_h)f_{i+1},\\
&C_v^{(i,i+1)} \equiv d^{(i)}_v (f_i-f_{i+1}),\\
&Z_h^{(i+1)} \equiv Z_h^{(i)}-\alpha_{i+1} f_{i+1},\\
&C^{(i+1)} \equiv \sigma^* C^{(i)}-(C^{(0)}\cdot L_0)_{F^{(0)}}f_{i+1},
\end{align*}
where $\alpha_i\leqslant 2n^2$.
\end{Lemma}
\begin{proof}
The equalities follow from Lemma 2.13 and computations of intersections
\begin{align*}
&(C_h^{(i)}\cdot L_i)_{E^{(i)}}=\beta_i - 2d^{(i)}_h,\\
&(C_v^{(i)}\cdot L_i)_{E^{(i)}}=d^{(i)}_v,\\
&(C^{(i)}\cdot L_i)_{F^{(i)}}=(C^{(0)}\cdot L_0)_{F^{(0)}}.
\end{align*}
The bound on $\alpha_i$ follows from Lemma 2.14 and the equality $Z_h^{(i)}\cdot \sigma^*(F)=2n^2$.
\end{proof}

\begin{Lemma}
Vertical degrees $\beta_i$ and $d_v^{(i)}$ satisfy the following relations. 
For $i=1$
\begin{align*}
\beta_1+d_v^{(1)}=\alpha_1+(C^{(0)}\cdot L_0)-\lambda_1 \big( n-2\nu_Q(D) \big)-2\lambda_1^2,
\end{align*}
and for $i\geqslant2$
\begin{align*}
\beta_i+d_v^{(i)}=\alpha_{i}+(C^{(0)}\cdot L_0)-\lambda_i \big( n-2\nu_Q(D) \big) -2\lambda_i^2+d_v^{(i-1)}+\big(\beta_{i-1}-2d_h^{(i-1)}\big).
\end{align*}
\end{Lemma}

\begin{proof}
By Lemma 2.15 and Lemma 6.3
\begin{align*}
&z_1\equiv \sigma_1^{*}(z_0)+\lambda_1^2 \big( E^{(1)} \big)^2 -2\lambda_1(D^{(0)}_1\cdot L_{0})f\equiv \sigma_1^{*}(z_0)-\lambda_1^2\sigma_1^{*}(L_0)-\Big(\lambda_1\big(n-2\nu_Q(D)\big)+2\lambda_1^2 \Big)f_1.
\end{align*}
On the other hand the decomposition of $Z_1$ and Lemma 6.4 imply
\begin{align*}
&Z_1=Z_h^{(1)}+C^{(1)}+C_h^{(1)}+C_v^{(1)}+k_1L_1 \equiv \sigma_1^* \big(Z_h^{(0)}+C^{(0)}+k_1L_0 \big) + C_h^{(1)}+C_v^{(1)} - (\alpha_1 + C^{(0)}\cdot L_0)f_1.
\end{align*}
Combining these equivalences we find that the following holds modulo pullback of a cycle
\begin{align*}
(\beta_1+d_v^{(1)})f_1 \equiv C_h^{(1)}+C_v^{(1)} \equiv -\Big( \lambda_1  \big(n-2\nu_Q(D) \big)+2\lambda_1^2\Big)f_1+\big(\alpha_1 + C^{(0)}\cdot L_0\big)f_1.
\end{align*}

Similarly by Lemma 2.15 and Lemma 6.3
\begin{align*}
&z_i\equiv \sigma^{*}_i \big(Z^{(i-1)}\big)+\lambda_i^2\big(E^{(i)}\big)^2-2\lambda_i\big(D^{(i-1)}_1\cdot L_{i-1}\big)f\equiv\\
&\equiv\sigma_i^{*}\big(Z^{(i-1)}\big)-\lambda_i^2\sigma^{*}(s_{i-1})-\Big(\lambda_i\big(n-2\nu_Q(D)\big)-2\lambda_i^2\Big)f_i.
\end{align*}
Once again from the decomposition of $Z^{(i)}$ and Lemma 6.4 we see that
\begin{align*}
Z^{(i)}&=Z_h^{(i)}+C^{(i)}+\sigma^*C_h^{(1,2)}+\sigma^*C_v^{(1,2)}+\dots+C_h^{(i-1,i)}+C_v^{(i-1,i)}+C_h^{(i)}+C_v^{(i)}+k_iL_i\equiv\\
&\equiv\sigma^* ( \dots )+C_h^{(i)}+C_v^{(i)}-\Big(\alpha_{i}+C^{(0)}\cdot L_0+\big(\beta_{i-1} - 2d^{(i-1)}_h\big)+d^{(i-1)}_v\Big)f_i.
\end{align*}
Combining these equivalences and considering them modulo pullbacks of the cycles we conclude that
\begin{align*}
d^{(i)}_v+\beta_i = \Big(\alpha_{i}+C^{(0)}\cdot L_0+\big(\beta_i - 2d^{(i)}_h\big)&+d^{(i-1)}_v \Big)-\Big(\lambda_i\big(n-2\nu_Q(D)\big)+2\lambda_i^2 \Big).
\end{align*}
\end{proof}

\begin{Cor}
The vertical degrees are bounded as follows
\begin{align*}
\beta_i+d_v^{(i)}< \sum_{j=1}^i \big( 2n^2-2\lambda_i^2+n^2\gamma \big)
\end{align*}
\end{Cor}
\begin{proof}
The inequality $n\geqslant 2\nu_Q(D)$ holds since the pair $(X,\frac{1}{n}\mathcal{M})$ is canonical at $Q$. By Lemma 6.5 for $i=1$ 
\begin{align*}
\beta_1+d_v^{(1)}=\alpha_1+(C^{(0)}\cdot L_0)-\lambda_1\big(n-2\nu_Q(D)\big)-2\lambda_1^2.
\end{align*}
Combining it with the bounds $\alpha_1\leqslant 2n^2$ and $C^{(0)}\cdot L_0\leqslant C\cdot L\leqslant \gamma n^2$ we get
\begin{align*}
\beta_1+d_v^{(1)}<2n^2+\gamma n^2-2\lambda_1^2.
\end{align*}

Now suppose the inequality holds for $i-1$. Then using the same bounds we get
\begin{align*}
&\beta_i+d_v^{(i)}=\alpha_{i}+C^{(0)}\cdot L_0+d_v^{(i-1)}+\big(\beta_{i-1}-2d_h^{(i-1)}\big)-\\
&-\lambda_i\big(n-2\nu_Q(D)\big)-2\lambda_i^2\leqslant (2n^2+\gamma n^2 -2\lambda_i^2)+(d_v^{(i-1)}+\beta_{i-1}).
\end{align*}
\end{proof}

\begin{Cor}
\begin{enumerate}[(i)]
\item Let $B$ be a fiber of a ruled surface $E^{(i)}$ then 
\begin{align*}
\mult_B Z^{(i)} = \mult_B C^{(i)}_v \leqslant \sum_{j=1}^i \big( 2n^2-2\lambda_i^2+n^2\gamma \big)
\end{align*}
\item Let $B$ be a point on $E^{(i)}$ then
\begin{align*}
\mult_B (C^{(i)}_v + C^{(i)}_h) \leqslant \sum_{j=1}^i \big( 2n^2-2\lambda_i^2+n^2\gamma \big). 
\end{align*}

\end{enumerate}
\end{Cor}
\begin{proof}
Clearly $\mult_B C^{(i)}_v$ is bounded by a vertical degree $d^{(i)}_v$ whether $B$ is a point or a curve. Thus the inequality holds if $B$ is a curve.

Similarly, $\mult_B C^{(i)}_h\leqslant d^{(i)}_h$, hence $\mult_B (C^{(i)}_v + C^{(i)}_h) \leqslant d^{(i)}_v+d^{(i)}_h$. Since $C^{(i)}_h$ does not contain the exceptional section $d^{(i)}_h\leqslant \beta_i$. Therefore by Corollary 6.6 the inequalities hold.
\end{proof}
%
%Final proof section
%

\section{Supermaximal singularities upstairs}
In previous section we found an upper bound on the multiplicity of components of $Z^{(M)}$ at the center of $\nu$ on $X^{(M)}$. In this section we show that it contradicts Corti inequality.

\begin{Lemma}
The pair 
\begin{align*}
\bigg( X^{(M)},\frac{1}{n}\mathcal{M}^{(M)}-\Big(1-\frac{\lambda_1}{n}\Big)&E^{(1,M)}-\dots-\Big(M-\sum\frac{\lambda_i}{n}+M\gamma \Big)E^{(M)}-\gamma F^{(M)} \bigg)
\end{align*}
is strictly canonical at $\nu$.
\end{Lemma}
\begin{proof}
Since the dual graph of $E^{(i)}$ is a simple chain by Theorem 5.1
\begin{align*}
K_{X^{(M)}}-\sum_{i=1}^{M}iE^{(i)}-\frac{1}{2}E_Q\sim \sigma^*(K_X).
\end{align*}
We disregard $E_Q$ in equivalences since $E_Q$ is away from the center of $\nu$.
For a generic divisor $D\in\mathcal{M}$
\begin{align*}
D^{(M)}+\sum_{i=1}^M\Big(\sum_{j=1}^i \lambda_j\Big)E^{(i)}=\sigma^*(D)
\end{align*}
and by Theorem 5.1
\begin{align*}
F^{(M)}+\sum_{i=1}^{M}iE^{(i)}=\sigma^*(F).
\end{align*}
Thus the pair in the statement of the lemma is a log pullback of the pair
$\Big(X,\frac{1}{n}\mathcal{M}-\gamma F\Big)$. Hence by Remark 2.5 the pair is strictly canonical at $\nu$.
\end{proof}

Clearly $\nu$ and $X^{(0)}$ satisfy the requirements of Proposition 6.1, that is the center of $\nu$ on $X^{(M)}$ is not a point on the exceptional section of $E^{(M)}$. We consider the three possibilities for the center of $\nu$.

%CASE A

\subsection{Case A}
Suppose the center $B$ of $\nu$ on $X^{(M)}$ is a fiber of $E^{(M)}$. Then the only divisor in the boundary which contains $B$ is $E^{(M)}$. Thus the pair
\begin{align*}
\bigg( X^{(M)},\frac{1}{n}\mathcal{M}^{(M)}-\big(M-\sum\frac{\lambda_i}{n}+M\gamma\big)E^{(M)} \bigg)
\end{align*}
is strictly canonical at $\nu$. By Lemma 2.6
\begin{align*}
\mult_B Z^{(M)} \geqslant 4n^2\frac{Mn-\sum\lambda_i}{n}+4n^2M\gamma.
\end{align*}
Combining this inequality with Corollary 6.7 we get
\begin{align*}
2Mn^2+Mn^2\gamma-2\sum_{i=1}^M\lambda_i^2 > 4\big(Mn^2-n\sum\lambda_i+Mn^2\gamma \big)
\end{align*}
or, equivalently,
\begin{align*}
0>3Mn^2\gamma+2\sum_{i=1}^M \big( n^2-2n\lambda_i+\lambda_i^2 \big),
\end{align*}
contradiction.

%CASE B

\subsection{Case B}
Suppose the center $B$ of $\nu$ on $X^{(M)}$ is a point which is not on $E^{(M-1)}$. Then the only divisor in the boundary containing $B$ is $E^{(M)}$. Thus the pair
\begin{align*}
\bigg( X^{(M)},\frac{1}{n}\mathcal{M}^{(M)}-\big(M-\sum\frac{\lambda_i}{n}+M\gamma\big)E^{(M)} \bigg)
\end{align*}
is strictly canonical at $\nu$. The components of $Z^{(M)}$ which may pass through $B$ are $Z_h^{(M)}$, $C^{(M)}_v$, and $C^{(M)}_h$. By Corti inequality there is a number $0< t\leqslant 1$ such that
\begin{align*}
\mult_B Z_h^{(M)} + t\mult_B & \big(C^{(M)}_v + C^{(M)}_h\big) >4n^2\Big(1+t\frac{Mn-\sum\lambda_i}{n}+Mtn^2\gamma \Big)=\\
=&4n^2+4tMn^2+4tMn^2\gamma-4tn\sum\lambda_i.
\end{align*}
On the other hand $\mult_{B} Z_h^{(M)}\leqslant Z_h\cdot F= 2n^2$ and we have a bound on $\mult_{B} (C_h^{(M)}+C_v^{(M)})$ by Corollary 6.7. Combining the bounds we get
\begin{align*}
2n^2+2tMn^2+tMn^2\gamma-2t\sum_{i=1}^M \lambda_i^2>4n^2+4tMn^2+4tMn^2\gamma-4tn\sum_{i=1}^M\lambda_i.
\end{align*}
Rearranging the terms we find an equivalent inequality:
\begin{align*}
0>2n^2+3tMn^2\gamma+2t\sum_{i=1}^M (n-\lambda_i)^2,
\end{align*}
contradiction.

%CASE C
\subsection{Case C}
Suppose the center $B$ of $\nu$ on $X^{(M)}$ is a point on the intersection $E^{(M)}\cap E^{(M-1)}$. Clearly these are the only divisors of the boundary containing $B$. Let $M^-=M-1$ for the compactness of formulas. Then the pair
\begin{align*}
\bigg( X^{(M)},\frac{1}{n}\mathcal{M}^{(M)}-& \big( M-\sum\frac{\lambda_i}{n}+M\gamma \big) E^{(M)}-\big(M^--\sum\frac{\lambda_i}{n}+M^-\gamma \big)E^{(M^-)}\bigg)
\end{align*}
is strictly canonical at $\nu$. Hence, compared to the last case there are $2$ more cycles which may contain $B$: $C^{(M^-,M}_h$ and $C^{(M^-,M)}_v$. By Corti inequality there are numbers $0<t,t^-\leqslant 1$ such that
\begin{align*}
&\mult_B Z_h^{(M)}+t \mult_B \big(C^{(M)}_v + C^{(M)}_h \big)+t^-\mult_B \big(C^{(M^-,M)}_v + C^{(M^-,M)}_h \big)\geqslant\\
\geqslant& 4n^2 + 4tMn^2+4tMn^2\gamma-4tn\sum_{i=1}^M\lambda_i
+4t^-M^-n^2+4t^-M^-n^2\gamma-4t^-n\sum_{i=1}^{M^{-}}\lambda_i.
\end{align*}
On the other hand we have the bounds on multiplicities from Corollary 6.7. After combining the inequalities and rearranging the terms we get
\begin{align*}
0>2n^2 + 3tMn^2\gamma + 3t^-M^-n^2\gamma + 2t\sum_{i=1}^M (n-\lambda_i)^2 + 2t^-\sum_{i=1}^{M^{-}}(n-\lambda_i)^2,
\end{align*}
contradiction.

%
%Epilogue
%

\subsection{Epilogue}
\begin{proof}[Proof of Theorem 1.1]
Suppose $X$ is not birationally rigid, that is there is a birational map $g$ to a Mori fiber space which does not satisfy requirements of Definition 4.1. Then by Proposition 4.3 there is a system $\mathcal{M}\subset \big| -nK_X+\gamma nF \big|$ such that one of the conditions $(i)$ and $(ii)$ holds. 

Suppose the pair $(X,\frac{1}{n}\mathcal{M})$ is not canonical at a curve $C$. By Proposition 4.4 the curve $C$ is a section. Then by Proposition 4.5 there is a birational involution $\chi_C$ such that $g\circ\chi_C$ does not satisfy conditions of Definition 4.1. Now we consider the map $g\circ \chi_C$ instead of $g$ and the linear system 
\begin{align*}
\chi_C(\mathcal{M})\subset \big| -n^\prime K_X+n^\prime \gamma^\prime \big|
\end{align*}
instead of $\mathcal{M}$. We repeat this process until either $n^\prime=0$, that is the new map $g^\prime=g\circ \chi_{C_1}\circ\dots\circ\chi_{C_k}$ is fiberwise (contradiction), or until the pair, corresponding to $g^\prime$ is canonical at all curves.

Now we are in the situation $(ii)$ of Proposition 4.3. Hence by Corollary 4.9 there is a supermaximal singularity. By Proposition 4.10 and Proposition 4.4 its center is a nonsingular point in the fiber containing half-point. There is a unique half-line passing through the center of $\nu$. Let us consider the associated ladder $\sigma_i:X^{(i)}\to X^{(i-1)}$. Let $D^{(i)}_1$ and $D^{(i)}_2$ be the proper transforms on $X^{(i-1)}$ of generic divisors $D_1,D_2\in\mathcal{M}$ and let $Z_i=D^{(i)}\cdot D^{(i)}$. Then as we have shown in section 6, there are bounds on multiplicities on components of the cycle $Z_i$. In section 7 we have proven that these bounds contradict Corti inequality. Thus there are no supermaximal singularities and $X$ is birationally rigid. 
\end{proof}

\end{document}